\documentclass[reqno, a4paper, 11pt]{amsart}
\usepackage{amssymb,url, color, mathrsfs}
\usepackage[colorlinks=true, bookmarks=true, pdfstartview=FitH, pagebackref=true, linktocpage=true]{hyperref}
\usepackage[short,nodayofweek]{datetime}
\usepackage[pagewise,running,mathlines,displaymath, switch]{lineno}
\usepackage{mathabx}

\usepackage[a4paper, vmargin={30mm,30mm},hmargin={20mm,20mm}, top=2.5cm, bottom=2.5cm, left=3cm, right=3cm]{geometry}

\usepackage[T5]{fontenc}
\usepackage[utf8]{inputenc}

\theoremstyle{plain}
\numberwithin{equation}{section}
\newtheorem{theorem}{Theorem}[section] 
\newtheorem{lemma}{Lemma}[section] 

\newtheorem{proposition}{Proposition}[section]

\def\N{\mathbb N}
\def\C{\mathbb C} 
\def\Z{\mathbf Z} 

\def\bT{\mathbf T} 
\def\vk{\mathbf k}
\def\cS{\mathcal S} 

\def\bB{\mathscr B} 
\def\cE{\mathcal E} 
\def\Csf{\mathsf c} 

\def\sN{\mathsf N}
\def\sK{\mathsf M}

\author[Tiến-Tài Nguyễn]{Tiến-Tài Nguyễn }
\address[Tiến-Tài Nguyễn]{ University of Science, Vietnam National University,   Hanoi, Vietnam}
\email{\href{mailto: Tiến-Tài Nguyễn <nttai.hus@vnu.edu.vn>}{nttai.hus@vnu.edu.vn}}

\begin{document}
\allowdisplaybreaks

\setpagewiselinenumbers
\setlength\linenumbersep{100pt}

\title[Rayleigh-Taylor instability]{Influence of capillary number on nonlinear Rayleigh-Taylor instability to the Navier-Stokes-Korteweg equations}

\begin{abstract}
Motivated by Bresch, Desjardins, Gisclon and Sart \cite{BDGS08}, in this paper, we study the influence of capillary number on an instability result related to the Navier-Stokes-Korteweg equations. Precisely, we investigate the instability of  a steady-state profile with a heavier fluid lying above a lighter fluid, i.e., to study the Rayleigh–Taylor instability problem if   the capillary number is below the critical value. After writing the nonlinear equations in a perturbed form, the first part is to provide a spectral analysis showing that, there exist \textit{possibly multiple} normal modes to the linearized equations by following  the operator method of Lafitte-Nguyễn  \cite{LN20}.    Hence, we  construct a \textit{wide class} of initial data for which the nonlinear perturbation problem departs from the equilibrium, based on the finding of \textit{possibly multiple} normal modes.  Using a refined  framework of Guo-Strauss \cite{GS95}, we prove the nonlinear instability.
\end{abstract}

\date{\bf \today \; at \, \currenttime}



\maketitle

\section*{Acknowledgements}

This paper is  finalized while the author attends a summer school at Institut d'Etudes Scientifiques de Cargèse (IESC), Corsica-France,  in June, 2023. The author is grateful to the warm hospitality of IESC and  the financial support of Universit\'e Sorbonne Paris Nord during his stay.  The author  also wishes to thank Prof O. Lafitte and Prof. J. Rauch for their advice. 
 
\section{Introduction}

In 1883, Lord Rayleigh \cite{Str83} studied the linear stability of the eigenvalue problem for two layers of gravity-driven incompressible and inviscid fluids, the heavy one is on the top of the light one and addressed the general stability criterion. Rayleigh’s work was taken up by Taylor \cite{Tay50} in 1950, in a more general set-up considering the effect of any accelerating field.  This Rayleigh-Taylor (RT) instability appears and has attracted much attention  due to both  its physical and mathematical importance. For a detailed physical comprehension of the RT instability, we refer to the book of Chandrasekhar \cite{Cha61} and some physics reports \cite{Kull91, Zhou17_1, Zhou17_2}. Mathematically speaking,  the effect of physical parameters such as internal surface tension \cite{WT12}, magnetic field \cite{JJW16, Wan19} on the nonlinear RT instability has been widely studied.  In this paper,  we study the influence of capillary number on nonlinear instability of an increasing RT density profile. This work is  motivated by Bresch, Desjardins, Gisclon and Sart \cite{BDGS08}, where  they investigated the expression of the largest growth rate in a small regime of the characteristic length $L_0$ of  RT density profile (see $L_0$ in Lemma \ref{LemEigenvalueReal}) by following an asymptotic analysis initiated by Cherfils-Clerouin, Lafitte and Raviart \cite{CCLR01}. 

Let us describe  the formulation of the main problem. Let $\bT$ be the usual $1D$-torus , $L>0$ and $\Omega =(2\pi L\bT)^2\times (0,1)$. We are  concerned with the following Navier-Stokes-Korteweg equations, introduced firstly by Korteweg \cite{Kor01}, describing the dynamics of an incompressible viscous fluid endowed with internal capillarity
\begin{equation}\label{EqNSKor}
\begin{cases}
\partial_t \rho+ \text{div}(\rho u) =0,\\
\partial_t(\rho u)+ \text{div}( \rho u\otimes u) -\mu \Delta u +\sigma \nabla \rho  \Delta \rho + \nabla P =-g\rho  e_3 ,\\
\text{div} u=0.
\end{cases}
\end{equation}
where $t\geqslant 0, x=(x_1,x_2,x_3)\in  \Omega$. The unknowns $\rho := \rho(x,t)$, $u :=  u(x,t)$ and $P :=P(x,t)$ denote the density, the velocity and the pressure of the fluid, respectively.   $e_3=(0,0,1)^T$ is the vertical unit vector. The parameter $\sigma>0$ is the capillary coefficient, $\mu>0$ is the viscosity coefficient, and $g > 0$ is the gravity constant. 

Denote $'=d/dx_3$ and $I=(0,1)$ Let $\rho_0$ and $P_0$ be two functions depending on $x_3\in I$ such that 
\begin{equation}\label{RelationRhoP}
P_0'= -\sigma \rho_0' \rho_0'' - g\rho_0.
\end{equation}
Hence, $(\rho, u, P)(t,x) = (\rho_0, 0, P_0)(x_3)$ is a steady-state of  Eq. \eqref{EqNSKor}. Let us define the perturbations 
\begin{equation}
\theta = \rho - \rho_0, \quad u=u-0, \quad q=P-P_0,
\end{equation}
and write Eq. \eqref{EqNSKor} in the following perturbed form
\begin{equation}\label{EqNSK_Pertur}
\begin{cases}
\partial_t \theta + \rho_0' u_3 = -u\cdot\nabla\theta,\\
\rho_0 \partial_t u +\nabla q -\mu\Delta u +\sigma( \rho_0' \Delta\theta e_3 + \rho_0'' \nabla\theta)+ g\theta e_3 \\
\qquad  =-\theta\partial_t u -(\rho_0+\theta) u\cdot \nabla u - \sigma \nabla\theta \Delta\theta,\\
\text{div} u=0.
\end{cases}
\end{equation}
Let us specify the initial data
\begin{equation}\label{InitialData}
(\theta, u)|_{t=0}= (\theta_0, u_0) \text{ in }\Omega,
\end{equation}
and the boundary conditions
\begin{equation}\label{BoundLinearized}
u|_{\partial \Omega}=0 \quad\text{for any }t>0.
\end{equation}
The initial data should satisfy the compatibility condition $\text{div}u_0 = 0$.  

The Rayleigh-Taylor problem is to study the stability of the equilibrium $(\rho_0, 0, P_0)(x_3)$ to the nonlinear equations \eqref{EqNSKor}, i.e of  the stability of the trivial equilibrium to the nonlinear equations \eqref{EqNSK_Pertur} where the density profile $\rho_0$ satisfies
\begin{equation}\label{RhoPositive}
\rho_0 \in C^\infty(I), \quad \min_I \rho_0 >0, \quad  \min_I \rho_0' >0.
\end{equation}
Let  us define the critical capillary number 
\begin{equation}\label{CapCritical}
\sigma_c := \sup_{\vartheta \in H_0^1(I)} \frac{g\int_I \rho_0' \vartheta^2}{\int_I (\rho_0')^2 (\vartheta')^2} \in (0,+\infty).
\end{equation}
Note that $\sigma_c$ is positive and finite due to the assumption \eqref{RhoPositive}.  As $\sigma <\sigma_c$, we first present the normal mode ansatz of the linearized equations \eqref{EqLinearized} showing the existence of \textit{possibly multiple} normal mode solutions, see Theorems \ref{ThmModeNormal}, \ref{ThmLinear} in Section \ref{SectMainResults}, thanks to Lafitte-Nguyễn's operator method \cite{LN20}. Section \ref{SectLinear} is devoted to the proof of the linear theorems. Once the linear instability is proven, we move to show the nonlinear RT instability, see Theorem \ref{ThmUnstable} in Section \ref{SectNonlinear}. After establishing \textit{a priori} energy estimates in Section \ref{SectEnergyEstimates}, we will give  the proof of  nonlinear instability in Section \ref{SectNonlinearInstability}, which is in the same spirit of the RT instability problem with Navier-slip boundary conditions \cite{Tai22}. 

In this work,  we are not only concerned with proving the nonlinear instability  in the regime $\sigma < \sigma_c$, that is showing that there exists at least one initial value  for which an instability develops as shown by Guo-Strauss \cite{GS95} (see also \cite{Gre00}), but we are able to prove a more general result on a \textit{wide class} of initial data, based on the existence of \textit{possibly multiple} normal mode solutions to the linearized equations.

We remark that  the recent paper of Zhang \cite{Zha22} and  of Li-Zhang \cite{LZ23} only prove the nonlinear RT instability in a small regime of  capillary number, i.e. $0<\sigma\ll 1$. Our nonlinear result shows that $0<\sigma <\sigma_c$ is also the subcritical regime of nonlinear RT instability.

\noindent{\textbf{Notations}.} We use  the notation $a\lesssim b$ to mean that $a\leq Cb$ for a universal constant $C > 0$,  which depends on the parameters of the problem and does not depend on the data. Throughout this paper, we write $H^s$ instead of $H^s(\Omega)$ for $s\geq 0$ and $\int$ instead of $\int_\Omega$. 

\section{The main results}\label{SectMainResults}

\subsection{The linear instability}
By omitting all nonlinear terms in \eqref{EqNSK_Pertur}, we obtain the linearized equations
\begin{equation}\label{EqLinearized}
\begin{cases}
\partial_t \theta + \rho_0' u_3 = 0,\\
\rho_0 \partial_t u +\nabla q -\mu\Delta u +\sigma( \rho_0' \Delta \theta e_3 +\rho_0'' \nabla\theta )+ g\theta e_3 = 0,\\
\text{div} u=0.
\end{cases}
\end{equation}
with the boundary condition \eqref{BoundLinearized}.  Following \cite{Cha61}, let $\vk=(k_1,k_2)\in (L^{-1}\Z)^2$ and in what follows, we always write $k=|\vk|= \sqrt{k_1^2+k_2^2}$, we look for normal mode solutions of Eq. \eqref{EqLinearized}-\eqref{BoundLinearized}, which are of the form 
\begin{equation}\label{NormalMode}
\begin{cases}
\theta(t,x)= e^{\lambda t} \cos(k_1x_1+k_2x_2) \eta(x_3), \\
u_1(t,x)= e^{\lambda t} \sin(k_1x_1+k_2x_2) v_1(x_3),\\
u_2(t,x)= e^{\lambda t} \sin(k_1x_1+k_2x_2) v_2(x_3),\\
u_3(t,x) = e^{\lambda t} \cos(k_1x_1+k_2x_2) \phi(x_3),\\
q(t,x) = e^{\lambda t}\cos(k_1x_1+k_2x_2) \pi(x_3).
\end{cases}
\end{equation}
In this situation, $\lambda \in \C$ with $\text{Re}\lambda>0$ is called as a characteristic value of the linearized equations \eqref{EqLinearized} after Chandrasekhar \cite{Cha61}.  Substituting \eqref{NormalMode} into Eq. \eqref{EqLinearized}-\eqref{BoundLinearized}, we obtain the ODE system in $(0,1)$, 
\begin{equation}\label{SystMode_1}
\begin{cases}
\lambda \eta+ \rho_0' \phi =0,\\
\lambda\rho_0 v_1- k_1 \pi - \mu (-k^2 v_1+v_1'')= \sigma k_1 \rho_0'' \eta, \\
\lambda\rho_0 v_2 - k_2 \pi - \mu (-k^2 v_2+v_2'')= \sigma k_2 \rho_0'' \eta,\\
\lambda\rho_0 \phi +\pi' -\mu (-k^2\phi+\phi'') = -\sigma \big( \rho_0' (-k^2\eta+\eta'')+ \rho_0''\eta' \big) -g\eta,\\
k_1v_1+k_2 v_2+\phi'=0.
\end{cases}
\end{equation}
withe the boundary conditions
\begin{equation}
v_1(0)=v_2(0)=\phi(0)=0, \quad\text{and } v_1(1)=v_2(1)=\phi(1)=0.
\end{equation}
Then eliminating  $\eta$ by using $\eqref{SystMode_1}_1$, we obtain 
\begin{equation}\label{SystMode_2}
\begin{cases}
-\lambda^2 \rho_0 v_1 +\lambda k_1 \pi -\lambda \mu (k^2 v_1-v_1'')= \sigma k_1\rho_0'\rho_0''\phi, \\
-\lambda^2 \rho_0 v_2 +\lambda k_2 \pi -\lambda \mu (k^2 v_2-v_2'')= \sigma k_2\rho_0'\rho_0''\phi,\\
\lambda^2 \rho_0\phi+\lambda \pi' + \lambda\mu(k^2\phi-\phi'')=- \sigma (\rho_0')^2 k^2 \phi + \sigma (\rho_0'(\rho_0'\phi)')' +g\rho_0'\phi,\\
k_1v_1+k_2 v_2+\phi'=0.
\end{cases}
\end{equation}
From two first equations of  \eqref{SystMode_2} and  $\eqref{SystMode_2}_4$ also, we have
\begin{equation}\label{EqPressure}
\pi = \frac1{\lambda k^2} ( -\lambda^2\rho_0 \phi'- \lambda\mu(k^2\phi'-\phi''') +\sigma k^2 \rho_0'\rho_0''\phi).
\end{equation}
Substituting $q$ from \eqref{EqPressure} from $\eqref{SystMode_2}_3$,  we arrive at a fourth-order ODE \begin{equation}\label{4thOrderEqPhi}
\lambda^2 (k^2\rho_0\phi - (\rho_0\phi')' )+ \lambda\mu (\phi^{(4)}-2k^2\phi''+k^4\phi) = gk^2\rho_0'\phi + \sigma k^2 ((\rho_0')^2\phi')'  -\sigma k^4(\rho_0')^2\phi,
\end{equation}
with the boundary conditions 
\begin{equation}\label{BoundODE}
\phi(0)=\phi'(0)=0, \quad\text{and } \phi(1)=\phi'(1)=0.
\end{equation}
Necessarily, we have:
\begin{lemma}\label{LemEigenvalueReal}
All characteristic values $\lambda$ are real and uniformly  bounded in $\vk$ by $\sqrt{\frac{g}{L_0}}$, where $L_0^{-1} := \max_I \frac{\rho_0'}{\rho_0}$ is the characteristic length of density profile.
\end{lemma}
Since all characteristic values $\lambda$ are real, we restrict to real-valued functions in the linear analysis.  As $k$ being fixed, we state the following $k$-subcritical regime of capillary number to investigate Eq. \eqref{4thOrderEqPhi}-\eqref{BoundODE}, thanks to the operator method  initiated by Lafitte and Nguyễn \cite{LN20}.
 
 \begin{theorem}\label{ThmModeNormal}
Let $\rho_0$ satisfy \eqref{RhoPositive} and $k$ be fixed. We define 
 \begin{equation}\label{kCapCritical}
\sigma_c(k) :=  \sup_{\vartheta \in H_0^1(I)} \frac{g\int_I \rho_0' \vartheta^2}{\int_I (\rho_0')^2 (k^2 \vartheta^2+ (\vartheta')^2)} \in (0,\sigma_c).
 \end{equation}
Hence, for all  $0<\sigma <\sigma_c(k)$,  there exists a finite sequence of real characteristic values
\[
+\infty > \lambda_1(k,\sigma) >\lambda_2(k,\sigma) > \dots > \lambda_\sN(k,\sigma) >0
\]
such that for each $\lambda_j$,  there is a  smooth solution $\phi_j \in H_0^\infty(I)$ to \eqref{4thOrderEqPhi}-\eqref{BoundODE}  as $\lambda=\lambda_j$.
\end{theorem}
It can be seen that $\sigma_c(k)$ is decreasing for $k\in(0,+\infty)$ and $\sigma_c(k) \nearrow \sigma_c$ as $k \searrow 0$. Hence, for each $\sigma <\sigma_c$, we define the set
\[
\cS := \{ \vk  \in (L^{-1}\Z)^2\{0\}: \sigma<\sigma_c(k)\} \neq \emptyset.
\]
As a result of Theorem \ref{ThmModeNormal}, we obtain our next theorem, showing  \textit{possibly multiple}  normal mode solutions \eqref{NormalMode} to the linearized equations \eqref{EqLinearized} for some wavenumber $\vk$.
\begin{theorem}\label{ThmLinear}
Let $\rho_0$ satisfy \eqref{RhoPositive}  and  $0<\sigma <\sigma_c$.  For each $\vk=(k_1,k_2)\in \cS$, the  linearized equations \eqref{EqLinearized}-\eqref{BoundLinearized} admit possibly multiple normal mode  solutions of the form $(1\leq j\leq\sN)$
\[
\begin{cases}
\theta_j(t,x)= e^{\lambda_j(\vk,\sigma) t} \cos(k_1x_1+k_2x_2) \eta_j(\vk,\sigma,x_3), \\
u_{1,j}(t,x)= e^{\lambda_j(\vk,\sigma) t} \sin(k_1x_1+k_2x_2) v_{1,j}(\vk,\sigma,x_3),\\
u_{2,j}(t,x)= e^{\lambda_j(\vk,\sigma) t} \sin(k_1x_1+k_2x_2) v_{2,j}(\vk,\sigma,x_3),\\
u_{3,j}(t,x) = e^{\lambda_j(\vk,\sigma) t} \cos(k_1x_1+k_2x_2) \phi_j(\vk,\sigma,x_3),\\
p_j(t,x) = e^{\lambda_j(\vk,\sigma) t}\cos(k_1x_1+k_2x_2) q_j(\vk,\sigma,x_3),
\end{cases}
\]
where $\eta_j, v_{1,j}, v_{2,j}, \phi_j$ and $q_j$ are real-valued and smooth functions. 
\end{theorem}
Note from Lemma \ref{LemEigenvalueReal} that 
\begin{equation}\label{DefineLambda}
0<\Lambda :=\sup_{\vk \in \cS } \lambda_1(\vk,\sigma) \leq \sqrt\frac{g}{L_0},
\end{equation}
we show that $\Lambda$ is the maximal growth rate of the linearized equations, see Proposition \ref{PropSharpGrowthRate}, to end the linear analysis section.

\subsection{Nonlinear instability}

Let us consider now the nonlinear equations \eqref{EqNSK_Pertur}. We begin with the \textit{a priori} energy estimate in Section \ref{SectEnergyEstimates} (see Proposition \ref{PropEstimates}). After that,   we prove the nonlinear instability  in Section \ref{SectNonlinearInstability}.

As $\sigma<\sigma_c$, we obtain possibly multiple normal mode solutions $(\theta_j, u_j, p_j)$ of \eqref{EqLinearized}-\eqref{BoundLinearized} for each $\vk \in \cS$ from Theorem \ref{ThmLinear}. Let 
\[
\cS_\Lambda := \{\vk \in \cS: \lambda_1(\vk,\sigma) > \frac23\Lambda\}.
\]
Hence, we define uniquely $1\leq \sK \leq \sN$  such that
\begin{equation}\label{AssumeLambda}
\Lambda > \lambda_1(\vk,\sigma) > \lambda_2(\vk,\sigma)> \dots > \lambda_\sK(\vk,\sigma) > \frac23 \Lambda > \lambda_{\sK+1}(\vk,\sigma) > \dots>\lambda_\sN(\vk,\sigma). 
\end{equation}
Let us fix $\sigma\in (0,\sigma_c)$ and $\vk \in \cS_\Lambda$, we consider a linear combination of normal modes 
\[
(\theta^\sN, u^\sN, q^\sN)(t,x)=  \sum_{j=1}^\sN \Csf_j (\theta_j, u_j, \pi_j)(t,x) \quad(\text{some } \Csf_j \text{ can be zero})
\]
to be an approximate solution to the nonlinear equations \eqref{EqNSK_Pertur}, with constants $\Csf_j$ being chosen such that
\begin{equation}\label{1stCondC}
\text{at least one of }\Csf_j (1\leq j\leq\sK) \text{ is nonzero}
\end{equation}
and 
\begin{equation}\label{2ndCondC}
\frac12 |\Csf_{j_m}|\| u_{j_m}\|_{L^2} > \sum_{j\geq j_m+1}|\Csf_j| \|u_j \|_{L^2}\geq 0, \quad\text{where } j_m :=\min \{1\leq j\leq \sK, \Csf_j \neq 0\}. 
\end{equation}
Hence, let $\delta \in (0,1)$ be given and let $(\theta^\delta, u^\delta, p^\delta)(t,x)$ be a local solution of the nonlinear equations \eqref{EqNSK_Pertur} with the initial datum 
\begin{equation}\label{InitialData}
\delta  (\theta^\sN, u^\sN, q^\sN)(0,x) = \delta \sum_{j=1}^\sN \Csf_j (\theta_j, u_j, \pi_j)(0,x).
\end{equation}
 We now define the difference functions
\[
(\theta^d, u^d, q^d)(t,x) = (\theta^\delta, u^\delta, q^\delta)(t,x) -\delta  (\theta^\sN, u^\sN, q^\sN)(t,x)
\]
satisfying the following nonlinear equations
\begin{equation}\label{EqDiff}
\begin{cases}
\partial_t \theta^d + \rho_0' u_3^d = - u^\delta \cdot\nabla\theta^\delta, \\
\rho_0 \partial_t u^d + \nabla p^d-\mu\Delta u^d + \sigma (\rho_0'\Delta \theta^d e_3+\rho_0'' \nabla\theta^d) +g\theta^d e_3 \\
\qquad\qquad = -\theta^\delta  \partial_t u^\delta -(\rho_0+\theta^\delta) u^\delta\cdot \nabla u^\delta - \sigma \nabla \theta^\delta  \Delta\theta^\delta ,\\
\text{div} u^d=0
\end{cases}
\end{equation}
with the initial data $(\theta^d, u^d)=0$. 
By exploiting some energy estimates of Eq. \eqref{EqDiff} and  the \textit{a priori} energy estimate established in Proposition \ref{PropEstimates}, we deduce the bound  of $ \|(\theta^d, u^d)(t)\|_{L^2}^2$, for $t$ small enough (see Proposition \ref{PropBoundDiff}).
The nonlinear result thus follows

\begin{theorem}\label{ThmUnstable}
Let $\rho_0$ satisfy \eqref{RhoPositive} and let $0<\sigma <\sigma_c$.  There exist positive constants $\delta_0, \varepsilon_0$ sufficiently small and another constant $m_0>0$ such that for any $\delta \in (0,\delta_0)$, the nonlinear equations \eqref{EqNSK_Pertur}  with the initial datum $\delta (\theta^\sN, u^\sN,q^\sN)(0,x)$ of form \eqref{InitialData}  admits  a local solution $(\theta^\delta, u^\delta)$ satisfying
\begin{equation}\label{BoundU_2Tdelta}
\| u^\delta(T^\delta)\|_{L^2} \geq \delta \| u^\sN(T^\delta) \|_{L^2} - \|(u^\delta-\delta u^\sN)(T^\delta)\|_{L^2} \geq   m_0\varepsilon_0,
\end{equation}
where $T^\delta$ satisfies uniquely $\delta \sum_{j=1}^\sN |\Csf_j| e^{\lambda_j T^\delta} =\varepsilon_0$.
\end{theorem}


\section{Linear instability}\label{SectLinear}

The aim of this section is to prove the linear instability thanks to an operator method of Lafitte and Nguyễn \cite{LN20}. Let us prove Lemma \ref{LemEigenvalueReal} first. In the next steps,  we introduce some operators and study their spectrum to prove  Theorem \ref{ThmModeNormal} and Theorem \ref{ThmLinear}.

\begin{proof}[Proof of Lemma \ref{LemEigenvalueReal}]
Multiplying by $\overline\phi$ on both sides of \eqref{4thOrderEqPhi} and then integrating by parts, we  obtain that 
\[
\begin{split}
&\lambda^2\Big( \int_I ( k^2 \rho_0 |\phi|^2 + \rho_0 |\phi'|^2 ) - \rho_0\phi' \overline \phi\Big|_0^1 \Big)  + \lambda \mu  \int_I (|\phi''|^2+2k^2|\phi'|^2+k^4|\phi|^2 )\\
&\qquad +\lambda\mu ( \phi'''\overline\phi - \phi''\overline\phi' -2k^2  \phi'\overline\phi)\Big|_0^1\\
&=gk^2\int_I\rho_0'\phi^2 -\sigma k^2 \int_I (\rho_0')^2 (\phi')^2 - \sigma k^4\int_I (\rho_0')^2 \phi^2 +\sigma k^2 (\rho_0')^2\phi' \overline\phi\Big|_0^1.
\end{split}
\]
Using \eqref{BoundODE}, we get 
\begin{equation}\label{EqVariational}
\begin{split}
&\lambda^2 \int_I ( k^2 \rho_0 \phi^2 + \rho_0 (\phi')^2 )   + \lambda \mu  \int_I ((\phi'')^2+2k^2(\phi')^2+k^4\phi^2 ) \\
& =gk^2\int_I\rho_0'\phi^2 -\sigma k^2 \int_I (\rho_0')^2 (\phi')^2 - \sigma k^4\int_I (\rho_0')^2 \phi^2.
\end{split}
\end{equation}
Suppose that $\lambda = \lambda_1 + i\lambda_2$, then one deduces from \eqref{EqVariational} that 
\begin{equation}\label{EqImaginaryPart}
-2\lambda_1 \lambda_2 \int_I ( k^2 \rho_0 |\phi|^2 + \rho_0 |\phi'|^2 )  = \lambda_2\mu\int_I ( |\phi''|^2  + 2k^2 |\phi'|^2 +k^4|\phi|^2 ) .
\end{equation}
If $\lambda_2 \neq 0$, \eqref{EqImaginaryPart} leads us to
\[
-2\lambda_1 \int_I ( k^2 \rho_0 |\phi|^2 + \rho_0 |\phi'|^2 )  =  \mu \int_I ( |\phi''|^2 + 2k^2 |\phi'|^2+k^4|\phi |^2 )  <0,
\]
that contradiction yields $\lambda_2=0$, i.e. $\lambda$ is real.
Using \eqref{EqVariational} again, we further get that 
\[
\lambda^2 \int_I \rho_0(k^2|\phi|^2+|\phi'|^2)  \leqslant gk^2 \int_I \rho_0'|\phi|^2 -\sigma k^2 \int_I (\rho_0')^2 |\phi'|^2 - \sigma k^4\int_I (\rho_0')^2 |\phi|^2.
\]
It tells us that $\lambda$ is  bounded by  $\sqrt{\frac{g}{L_0}}$. This finishes the proof of Lemma \ref{LemEigenvalueReal}.
\end{proof}
\subsection{Auxiliary operators}
\begin{proposition}\label{PropOpe_Q}
The operator 
\[
Q_{k,\sigma} \vartheta := gk^2\rho_0'\vartheta + \sigma k^2 ((\rho_0')^2\vartheta')'  -\sigma k^4 (\rho_0')^2 \vartheta
\]
 from $H_0^1(I)\cap H^2(I)$ to $L^2(I)$ is symmetric. 
\end{proposition}
The proof of Proposition \ref{PropOpe_Q} is due to  direct computations via  integration by parts, that we omit the details. 

\begin{proposition}\label{PropOpe_P}
Let us define the bilinear form on $H_0^2(I)$ as follows, 
\begin{equation}\label{BilinearForm}
\bB_{k,\lambda}(\vartheta,\varrho) := \lambda \int_I \rho_0(k^2\vartheta\varrho + \vartheta'\varrho') + \mu \int_I (\vartheta''\varrho'' +2k^2 \vartheta'\varrho'+ k^4\vartheta\varrho).  
\end{equation}
We have that $\bB_{k,\lambda}$ is a continuous and coercive bilinear form on $H_0^2(I)$. Hence, there exists a unique operator $P_{k,\lambda}$, that is also bijective,  such that for all $\varrho \in H_0^2(I)$, we have
\begin{equation}\label{EqVariationalForm}
\bB_{k,\lambda}(\vartheta,\varrho) = \langle P_{k,\lambda} \vartheta, \varrho \rangle. 
\end{equation}
Furthermore, for any given $f\in L^2$, there exists a unique function $u\in H_0^2(I)\cap H^4(I)$ such that 
\begin{equation}\label{FormulaP}
P_{k,\lambda} \vartheta = \lambda (k^2\rho_0 \vartheta - (\rho_0 \vartheta')' )+ \mu ( \vartheta^{(4)}-2k^2 \vartheta''+k^4 \vartheta) = f.
\end{equation} 
\end{proposition}
\begin{proof}
The proof is straightforward thanks to Riesz's representation theorem, so we omit the details.
\end{proof}

Thanks to Propositions \ref{PropOpe_Q}, \ref{PropOpe_P}, we obtain the following proposition. 
\begin{proposition}\label{PropOpe_S}
The operator $S_{k,\lambda,\sigma}:=  P_{k,\lambda}^{-1/2}Q_{k,\sigma} P_{k,\lambda}^{-1/2}$ is compact and self-adjoint from $L^2(I)$ to itself.
\end{proposition}
\begin{proof}
Proposition \ref{PropOpe_P} helps us to define the inverse operator $P_{k,\lambda}^{-1}$ of $P_{k,\lambda}$, from $L^2(I)$ to $H_0^2(I)\cap H^4(I)$. Hence, let $\psi \in L^2(I)$, one has $P_{k,\lambda}^{-1/2}\psi$ belongs to  $H_0^2(I)$, yielding that $Q_{k,\sigma} P_{k,\lambda}^{-1/2}\psi \in L^2(I)$. We deduce that  $S_{k,\lambda,\sigma}$ sends $L^2(I)$ to $ H_0^2(I)$. Composing $S_{k,\lambda,\sigma}$ with the continuous injection $H^p(I)\hookrightarrow H^q(I)$ for $p>q\geq 0$, we obtain the  compactness and self-adjointness of $S_{k,\lambda,\sigma}$. The proof of Proposition \ref{PropOpe_S} is complete.
\end{proof}
Thanks to  the spectral theory of compact and self-adjoint operators again, we have that the discrete spectrum of the operator $S_{k,\lambda,\sigma}$ is an infinite sequence of eigenvalues, denoted by $\{\gamma_n= \gamma_n(k,\lambda,\sigma)\}_{n\geq 1}$, tending to 0 as $n\to\infty$.  We further obtain the following property of the largest eigenvalue $\gamma_1$. 
\begin{proposition}\label{PropGamma1}
Let us recall the bilinear form $\bB_{k,\lambda}$ \eqref{BilinearForm} and the critical capillary number \eqref{CapCritical}. For $0<\sigma<\sigma_c(k)$, there holds 
\begin{equation}\label{VariationalGamma1}
\frac{\gamma_1(k,\lambda,\sigma)}{k^2} = \max_{\vartheta \in  H_0^2(I)} \frac{g\int_I \rho_0'\vartheta^2 - \sigma \int_I (\rho_0')^2(k^2\vartheta^2 +(\vartheta')^2)}{\bB_{k,\lambda}(\vartheta,\vartheta)}>0.
\end{equation}
\end{proposition}
\begin{proof}
Since  the definition of $\sigma_c(k)$ \eqref{kCapCritical}, for $\sigma \in (0,\sigma_c(k))$, there exists a function $\vartheta \in H_0^1(I)$ such that 
\[
g\int_I \rho_0'\vartheta^2 - \sigma \int_I (\rho_0')^2(k^2\vartheta^2 +(\vartheta')^2)>0,
\]
it yields the positivity of 
\[
\max_{\vartheta \in  H_0^2(I)} \frac{g\int_I \rho_0'\vartheta^2 - \sigma \int_I (\rho_0')^2(k^2\vartheta^2 +(\vartheta')^2)}{\bB_{k,\lambda}(\vartheta,\vartheta)}.
\]
We now prove
\begin{equation}\label{IneGamma1Gtr}
\max_{\vartheta \in  H_0^2(I)} \frac{g\int_I \rho_0'\vartheta^2 - \sigma \int_I (\rho_0')^2(k^2\vartheta^2 +(\vartheta')^2)}{\bB_{k,\lambda}(\vartheta,\vartheta)} \leq  \frac{\gamma_1}{k^2}. 
\end{equation}
Let us consider the Lagrangian functional 
\[
\mathcal L_{k,\lambda,\sigma}(\alpha, \vartheta) = g\int_I \rho_0'\vartheta^2 - \sigma \int_I (\rho_0')^2(k^2\vartheta^2 +(\vartheta')^2) - \alpha(\bB_{k,\lambda}(\vartheta,\vartheta) -1).
\]
Thanks to Lagrange multiplier theorem, the extrema  of the quotient 
\[
\frac{g\int_I \rho_0'\vartheta^2 - \sigma \int_I (\rho_0')^2(k^2\vartheta^2 +(\vartheta')^2)}{\bB_{k,\lambda}(\vartheta,\vartheta)}
 \]
are necessarily the stationary points $(\alpha_\star, \vartheta_\star)$ of  $\mathcal L_{k,\lambda,\sigma}$, which satisfy that for all $\varrho \in H_0^2(I)$, 
\begin{equation}\label{EqWeakFormVariational}
\begin{split}
g \int_I \rho_0'  \vartheta_\star \varrho- \sigma \int_I  (\rho_0')^2 \vartheta_\star' \varrho'  -\sigma k^2\int_I  (\rho_0')^2 \vartheta \varrho - \alpha \bB_{k,\lambda}(\vartheta_\star,\varrho)=0
\end{split}
\end{equation}
and that
\begin{equation}\label{EqNormalizedU}
\bB_{k,\lambda}(\vartheta_\star,\vartheta_\star)=1.
\end{equation}
Owing to a bootstrap argument, we obtain from \eqref{EqWeakFormVariational} that $\vartheta_\star \in H_0^2(I)\cap H^4(I)$ is a solution of  $Q_{k,\sigma}\vartheta_\star= \alpha k^2 P_{k,\lambda} \vartheta_\star$ being normalized by \eqref{EqNormalizedU}. Hence, $\alpha_\star k^2$ is an eigenvalue of the compact and self-adjoint operator $S_{k,\lambda,\sigma} =P_{k,\lambda}^{-1/2} Q_{k,\sigma} P_{k,\lambda}^{-1/2}$ with  $P_{k,\lambda}^{-1/2}\vartheta_\star$ being an associated eigenfunction. We deduce that $\alpha k^2\leq \gamma_1(k,\lambda,\sigma)$, i.e. \eqref{IneGamma1Gtr}.

Next, we  prove  the reverse inequality 
\begin{equation}\label{IneGamma1Less}
\frac{\gamma_1}{k^2} \leq \max_{\vartheta \in  H_0^2(I)} \frac{g\int_I \rho_0'\vartheta^2 - \sigma \int_I (\rho_0')^2(k^2\vartheta^2 +(\vartheta')^2)}{\bB_{k,\lambda}(\vartheta,\vartheta)}
\end{equation}
For any $\psi \in L^2(I)$, there exists a unique $\vartheta \in H_0^2(I)$ such that $\vartheta =P_{k,\lambda}^{-1/2}\psi$. Hence
\[
\begin{split}
\frac{\langle S_{k,\lambda,\sigma} \psi,\psi \rangle}{\|\vartheta \|_{L^2(I)}^2} &= \frac{\langle Q_{k,\sigma} P_{k,\lambda}^{-1/2}\psi, P_{k,\lambda}^{-1/2}\psi\rangle}{\langle P_{k,\lambda} \psi, P_{k,\lambda}^{-1}\psi \rangle} =  \frac{\langle Q_{k,\sigma} P_{k,\lambda}^{-1/2}\psi, P_{k,\lambda}^{-1/2}\psi \rangle}{\langle P_{k,\lambda}(P_{k,\lambda}^{-1/2}\psi ),P_{k,\lambda}^{-1/2}\psi \rangle }= \frac{\langle Q_{k,\sigma} \vartheta,\vartheta \rangle}{\langle P_{k,\lambda} \vartheta,\vartheta \rangle},
\end{split}
\]
yielding 
\begin{equation}\label{EqSthetaQw}
\frac1{k^2} \frac{\langle S_{k,\lambda,\sigma} \psi,\psi \rangle}{\|\psi \|_{L^2(I)}^2} =\frac{g\int_I \rho_0'\vartheta^2 - \sigma \int_I (\rho_0')^2(k^2\vartheta^2 +(\vartheta')^2)}{\bB_{k,\lambda}(\vartheta,\vartheta)}.
\end{equation}
Meanwhile, since $S_{k,\lambda,\sigma}$ is a self-adjoint   operator, one has 
\begin{equation}\label{EqDefGamma_1}
\gamma_1= \sup_{\psi \in L^2(I)} \frac{\langle S_{k,\lambda,\sigma}\psi,\psi\rangle}{\|\psi\|_{L^2(I)}^2}.
\end{equation}
Combining \eqref{EqSthetaQw} and  \eqref{EqDefGamma_1}, it gives \eqref{IneGamma1Less}. In view of  \eqref{IneGamma1Gtr} and \eqref{IneGamma1Less}, we obtain \eqref{VariationalGamma1}.
\end{proof}

\begin{proposition}\label{PropPositive}
There exist finitely positive eigenvalues $\gamma_n$.
\end{proposition}
\begin{proof}
The operator $Q_{k,\sigma}$ can be seen as Dirichlet realization of a weighted Laplacian. Due to Poincar\'e's inequality, there exists a positive constant $C_0$ such that for all $\psi \in H_0^1(I)$, 
\[
\int_I \psi Q_{k,\sigma}\psi \leq C_0\int_I \psi^2. 
\]
That means  $Q_{k,\lambda}$ has finitely positive eigenvalues as $\sigma<\sigma_c(k)$. So does $S_{k,\lambda,\sigma}$. 
\end{proof}

Thanks to Propositions \ref{PropGamma1}, \ref{PropPositive}, we reorder the sequence $(\gamma_n(k,\lambda,\sigma))_{n\geq 1}$ as follows,
\begin{equation}\label{PositiveEigen}
\gamma_1(k,\lambda,\sigma)> \gamma_2(k,\lambda,\sigma)  > \dots > \gamma_\sN(k,\lambda,\sigma)  >0 > \gamma_{\sN+1}(k,\lambda,\sigma) > \dots,
\end{equation}
with 
\[
\lim_{j \to \infty} \gamma_{\sN+j}(k,\lambda,\sigma)=0.
\]

\subsection{Proof of  the linear instability}

Let $\psi_j=\psi_{j,k,\lambda,\sigma}\in L^2(I)$ be an eigenfunction of $S_{k,\lambda,\sigma}$ associated with the eigenvalue $\gamma_j$ $(1\leq j\leq \sN)$ listed above \eqref{PositiveEigen},  one has 
\[
S_{k,\lambda,\sigma} \psi_j= P_{k,\lambda}^{-1/2}Q_{k,\sigma} P_{k,\lambda}^{-1/2} \psi_j= \gamma_j^+\psi_j.
\]
This yields, $\phi_j=\phi_{j,k,\lambda,\sigma}= P_{k,\lambda}^{-1/2} \psi_j \in H_0^2(I)$ is a solution of 
\begin{equation}\label{EqQ_Gamma_Lambda}
Q_{k,\sigma}\phi_j = \gamma_j P_{k,\lambda} \phi_j.
\end{equation}
 For each $1\leq j\leq \sN$, in order to get $\phi_{j,k,\lambda,\sigma}$ is a solution of \eqref{4thOrderEqPhi}, it suffices to look for positive values of $\lambda_j$ such that 
\begin{equation}\label{EqFindLambda}
\gamma_j(k,\lambda_j,\sigma) =\lambda_j. 
\end{equation}
We state two lemmas to solve Eq. \eqref{EqFindLambda}.
\begin{lemma}\label{LemGammaCont}
We have that $\gamma_j(k,\lambda,\sigma)$ and $\psi_{j,k,\lambda,\sigma}$ are differentiable functions in $\lambda$.
\end{lemma}
The proof of Lemma \ref{LemGammaCont} is followed by the classical perturbation theory of the spectrum of operators of Kato \cite{Kato} and is the same as \cite[Lemma 3.2]{LN20}. Hence, we omit the details here. 

\begin{lemma}\label{LemGammaDecrease}
The function $\gamma_j(k,\lambda,\sigma)$ is  decreasing in $\lambda$. 
\end{lemma}
\begin{proof}
 Let $z_j=z_{j,k,\lambda,\sigma} = \frac{d}{d\lambda} \psi_{j,k,\lambda,\sigma}$, which enjoys 
\[
z_j(0)=z_j'(0)=z_j(1)=z_j'(1)=0.
\]
In view of \eqref{EqQ_Gamma_Lambda}, we get 
\[
\frac1{\gamma_j} Q_{k,\sigma} z_{j,\lambda,\sigma} + \frac{d}{d\lambda}\Big(\frac1{\gamma_j}  \Big) Q_{k,\sigma} \phi_j= P_{k,\lambda} z_j +\mu (\phi_j^{(4)}-2k^2\phi_j''+k^4\phi_j),
\]
Multiplying by $\phi_j$ on both sides of the resulting equation, we have 
\begin{equation}\label{1_GammaDecrease}
\begin{split}
&\frac1{\gamma_j} \langle Q_{k,\sigma} z_j, \phi_j \rangle+ \frac{d}{d\lambda}\Big(\frac1{\gamma_j}  \Big)   \langle Q_{k,\sigma} \phi_j,\phi_j \rangle = \langle P_{k,\lambda} z_j,\phi_j \rangle + \mu \int_I (\phi_j^{(4)}-2k^2\phi_j''+k^4\phi_j)\phi_j.
\end{split}
\end{equation}
Using  Propositions \ref{PropOpe_Q},  \ref{PropOpe_P}, we have
\begin{equation}\label{2_GammaDecrease}
\frac1{\gamma_j} \langle Q_{k,\sigma} z_j,\phi_j \rangle = \frac1{\gamma_j}\langle z_j, Q_{k,\sigma} \phi_j \rangle =   \langle z_j , P_{k,\lambda} \phi_j\rangle= \langle  P_{k,\lambda}z_j,  \phi_j\rangle.
\end{equation}
Substituting \eqref{2_GammaDecrease} into \eqref{1_GammaDecrease}, and using \eqref{EqQ_Gamma_Lambda}  again, we obtain 
\begin{equation}\label{3_GammaDecrease}
 \frac{d}{d\lambda}\Big(\frac1{\gamma_j}  \Big) \gamma_j   \langle P_{k,\lambda}\phi_j, \phi_j \rangle = \mu \int_I (\phi_j^{(4)}-2k^2\phi_j''+k^4\phi_j)\phi_j. \end{equation}
Thanks to the integration by parts and \eqref{EqVariationalForm}, we get further
\begin{equation}\label{4_GammaDecrease}
\frac{d}{d\lambda}\Big(\frac1{\gamma_j}  \Big)  \gamma_j^+ \bB_{k,\lambda}(\phi_j,\phi_j) =\mu  \int_I ((\phi_j'')^2+2k^2(\phi_j')^2+k^4\phi_j^2)>0.
 \end{equation}
It follows from \eqref{4_GammaDecrease}  that $\frac1{\gamma_j(k,\lambda,\sigma)}$ is  increasing in $\lambda$, i.e. $\gamma_j(k,\lambda,\sigma)$ is decreasing in $\lambda>0$.
\end{proof}

We are in position to prove Theorem \ref{ThmModeNormal}. 
\begin{proof}[Proof of Theorem \ref{ThmModeNormal}]
For each $j \in [1,\sN]$, we solve Eq. \eqref{EqFindLambda}. Since $\gamma_j(k,\lambda,\sigma)$ is a decreasing function in $\lambda$, we have $\gamma_j^+(k,\lambda,\sigma)> \gamma_j(k,\epsilon,\sigma)>0$ for any $0<\lambda\leq \epsilon$. This yields 
\begin{equation}\label{Limit0Gamma}
 \frac{\lambda}{\gamma_j(k,\lambda,\sigma)} \leq \frac{\lambda}{\gamma_j(k,\epsilon,\sigma)} \searrow 0 \quad\text{as } \lambda \searrow 0^+. 
\end{equation}
Meanwhile, using \eqref{EqVariationalForm} and \eqref{EqQ_Gamma_Lambda} again, we obtain 
\[
gk^2\int_I \rho_0' \phi_j^2 \geq \gamma_j^+(k,\lambda,\sigma)\Big( \lambda k^2 \int_I \rho_0 \phi_j^2 + \mu k^4 \int_I \phi_j^2\Big), 
\]
yielding 
\begin{equation}\label{LimitInftyGamma}
 \frac{\lambda}{\gamma_j(k,\lambda,\sigma)} \geq \frac{\lambda^2 \min_I\rho_0+\lambda\mu k^2 }{g\max_I\rho_0'} \nearrow +\infty \quad\text{as }\lambda\nearrow +\infty. 
\end{equation}
Owing to Lemma \ref{LemGammaDecrease} and two limits \eqref{Limit0Gamma} and \eqref{LimitInftyGamma}, there is a unique $\lambda_j=\lambda_j(k,\sigma) >0$ solving \eqref{EqFindLambda}. Hence, $\phi_j = \phi_{j,k,\lambda_j,\sigma}\in H_0^\infty(I)$ is a solution of \eqref{4thOrderEqPhi}-\eqref{BoundODE} as $\lambda=\lambda_j$, after a bootstrap argument. Note that, for all $1\leq j\leq \sN$,  we have $\lambda_j \in (0, \sqrt{\frac{g}{L_0}})$ since $\lambda_j$ is a characteristic value. 
Theorem \ref{ThmModeNormal} is proven.
\end{proof}

We now go back to the linearized equations \eqref{EqLinearized} and prove Theorem \ref{ThmLinear}.
\begin{proof}[Proof of Theorem \ref{ThmLinear}]
 Let us fix a wavenumber $\vk= (k_1,k_2) \in \mathcal S \cap (L^{-1}\Z)^2$ and deduce from Theorem \ref{ThmModeNormal} to obtain finitely or infinitely many characteristic values $\lambda_j(\sigma)$ $(1\leq j\leq \sN)$ and a smooth solution $\phi_{j,\sigma}$ of \eqref{4thOrderEqPhi}-\eqref{BoundODE} as $\lambda=\lambda_j(\sigma)$. Hence, in view if $\eqref{SystMode_1}_1$ and \eqref{EqPressure}, we define 
\[
\eta_j = -\frac{\rho_0' \phi_j}{\lambda_j} \quad\text{and}\quad q_j = \frac1{\lambda_jk^2}(-\lambda_j^2\rho_0\phi_j'-\lambda_j\mu(k^2\phi_j'-\phi_j''')+ \sigma k^2\rho_0'\rho_0'''\phi_j).
\]
Hence, we find $v_{1,j}$ as a solution of the second-order ODE on $(0,1)$ 
\[
-\lambda^2 \rho_0 v_1 +\lambda k_1 q_j -\lambda \mu (k^2 v_1-v_1'')= \sigma k_1\rho_0'\rho_0''\phi_j=0.
\]
with the boundary conditions $v_1(0)=v_1(1)=0$. Hence, define $v_{2,j}= -(k_1v_{1,j}+\phi_j)/k_2$, we conclude the proof of Theorem \ref{ThmLinear}. 
\end{proof}

\subsection{The maximal growth rate}

Letting $\lambda=\lambda_1$ in \eqref{VariationalGamma1}, we deduce the variational formulation of the largest characteristic value,
\begin{equation}\label{EqLambda_1}
\frac{\lambda_1}{k^2} = \max_{\vartheta\in H_0^2(I)} \frac{g\int_I\rho_0\vartheta^2 - \sigma \int_I(\rho_0')^2 (k^2\vartheta^2+ (\vartheta')^2)}{\bB_{k,\lambda_1}(\vartheta,\vartheta)}.
\end{equation}
In view of  \eqref{EqLambda_1} and the horizontal Fourier transform, we obtain the following lemma, in the same pattern as \cite[Lemma 4.1]{JJW16} and \cite[Lemma 4.1]{Zha22}.
\begin{lemma}\label{LemMaxLambda}
For any function $w\in H^1(\Omega)$ such that $\text{div}w=0$. There holds
\begin{equation}\label{LemLambda}
\int ( g\rho_0'|w|^2 - \sigma (\rho_0')^2 |\nabla w_3|^2 )   \leq \Lambda^2 \int\rho_0|w|^2  + \Lambda\mu \int\rho_0|\nabla w|^2 .
\end{equation}
\end{lemma}
We are in position to prove that $\Lambda$ is the maximal growth rate of the linearized equations \eqref{EqLinearized}-\eqref{BoundLinearized}. 
\begin{proposition}\label{PropSharpGrowthRate}
Let $(\theta,u,q)$ be a solution of the linearized equations \eqref{EqLinearized}-\eqref{BoundLinearized}, there holds
\begin{equation}\label{IneSharpGrowthRate}
\|(\theta,u)(t)\|_{L^2} \lesssim e^{\Lambda t} \|(\theta, u)(0)\|_{L^2}.
\end{equation}
\end{proposition}
\begin{proof}
We obtain from $\eqref{EqLinearized}_{1,2}$ that 
\[
\rho_0 \partial_t^2 u+\nabla\partial_t q-\mu\Delta\partial_t u = g\rho_0' u_3 + \sigma(\rho_0' \Delta(\rho_0'u_3) e_3 + \rho_0''\nabla(\rho_0'u_3)).
\]
That implies
\[\begin{split}
\frac12 \frac{d}{dt} \int \rho_0 |\partial_t u|^2 + \frac12 \int |\nabla\partial_t u|^2 &= g\int\rho_0' u_3\partial_t u_3 + \sigma \int \rho_0' \Delta(\rho_0'u_3)\partial_t u_3 + \sigma\int \rho_0'' \nabla(\rho_0'u_3) \cdot\partial_t u.
\end{split}
\]
Due to the following equalities,
\begin{equation}\label{LemLambda_5}
\begin{split}
\int\rho_0'  \Delta(\rho_0' u_3)\partial_t u_3 &= -\int\nabla(\rho_0'u_3)\cdot \nabla(\rho_0'\partial_t u_3)  = -\frac12 \frac{d}{dt} \int|\nabla(\rho_0' u_3)|^2, \\
\int\rho_0'' \nabla(\rho_0'u_3) \cdot \partial_t u &=-\int\rho_0' u_3 \text{div}(\rho_0''\partial_t u)  =-\frac12 \frac{d}{dt} \int\rho_0'\rho_0'''|u_3|^2,
\end{split}
\end{equation}
and
\begin{equation}\label{LemLambda_6}
\int |\nabla(\rho_0'u_3)|^2 + \int \rho_0'\rho_0''|u_3|^2 =\int (\rho_0')^2|\nabla u_3|^2,
\end{equation}
we get further 
\[
\frac12 \frac{d}{dt}\int(\rho_0|\partial_t u|^2 - g\rho_0'|u_3|^2 + \sigma (\rho_0')^2 |\nabla u_3|^2 ) + \mu \int\rho_0 |\nabla \partial_t u|^2=0.
\]
Together with \eqref{LemLambda}, we have
\begin{equation}\label{LemLambda_1}
\begin{split}
\|\sqrt{\rho_0}\partial_t u(t)\|_{L^2}^2 + 2\mu \int_0^t \|\nabla \partial_t u(\tau)\|_{L^2}^2 d\tau &= g\int \rho_0'|u_3(t)|^2 - \sigma \int(\rho_0')^2|\nabla u_3(t)|^2 \\
&\leq \Lambda^2 \|\sqrt{\rho_0} u(t)\|_{L^2}^2 +\Lambda\mu \|\nabla u\|_{L^2}^2.
\end{split}
\end{equation}
Meanwhile, we obtain
\begin{equation}\label{LemLambda_2}
 \partial_t\|\sqrt{\rho_0} u(t)\|_{L^2}^2 = 2\int\rho_0 u(t)\cdot \partial_t u(t) \leq \frac1{\Lambda} \|\sqrt{\rho_0}\partial_t u(t)\|_{L^2}^2+ \Lambda\|\sqrt{\rho_0} u(t)\|_{L^2}^2
\end{equation}
and 
\begin{equation}\label{LemLambda_3}
\Lambda\|\nabla u(t)\|_{L^2}^2 = 2\Lambda\int_0^t  \int \nabla \partial_t u(\tau) : \nabla u(\tau) d\tau \leq \int_0^t\|\nabla \partial_t u(\tau)\|_{L^2}^2 +\Lambda^2 \int_0^t\|\nabla u(t)\|_{L^2}^2 d\tau.
\end{equation}
Combining \eqref{LemLambda_1}, \eqref{LemLambda_2} and \eqref{LemLambda_3} gives us that 
\begin{equation}
\partial_t\|\sqrt{\rho_0}u(t)\|_{L^2}^2 + \mu\|\nabla u(t)\|_{L^2}^2 \leq 2\Lambda \Big( \|\sqrt{\rho_0} u(t)\|_{L^2}^2+ \mu \int_0^t\|\nabla u(t)\|_{L^2}^2 d\tau\Big).
\end{equation}
Applying Gronwall's inequality, we deduce 
\begin{equation}\label{LemLambda_4}
 \| u(t)\|_{L^2}^2+ \mu \int_0^t\|\nabla u(\tau)\|_{L^2}^2 d\tau \lesssim e^{2\Lambda t} \|u(0)\|_{L^2}^2.
 \end{equation}
Using $\eqref{EqLinearized}_1$ and \eqref{LemLambda_4}, we get
\[
\|\theta(t)\|_{L^2} \lesssim \|\theta(0)\|_{L^2}+ \int_0^t \|\partial_t\theta(\tau)\|_{L^2} d\tau \lesssim \|\theta(0)\|_{L^2} + \int_0^t \|u_3(\tau)\|_{L^2} d\tau \lesssim e^{\Lambda t} \|(\theta,u)(0)\|_{L^2}.
\]
The inequality \eqref{IneSharpGrowthRate} follows from the resulting inequality and \eqref{LemLambda_4}. Proof of Lemma \ref{PropSharpGrowthRate} is complete.
\end{proof}

\section{Nonlinear instability}\label{SectNonlinear}
\subsection{A priori energy estimates}\label{SectEnergyEstimates}

We refer to \cite{SBGLR06, TW10,  Has16, Wan17, JB23} to the local existence of regular solutions to the incompressible Navier-Stokes-Korteweg equations. Let $(\theta, u)(t)$ $(t\in [0,T^{\max}))$ be a local-in-time solution to the nonlinear equations \eqref{EqNSK_Pertur} with the initial data $(\theta, u)(0)$ such that 
\begin{equation}\label{AssumeSmallE}
\sup_{t\in [0,T^{\max})} \sqrt{\|\theta(t)\|_{H^3}^2 +\|u(t)\|_{H^3}^2} \leq \delta_0 \ll 1.
\end{equation}
The aim of this section is to demonstrate the following inequality.

\begin{proposition}\label{PropEstimates}
Let $\cE(t) := \sqrt{\|\theta(t)\|_{H^3}^2 +\|u(t)\|_{H^3}^2}>0$.  Under the smallness assumption \eqref{AssumeSmallE}. For any $\varepsilon>0$, there holds 
\begin{equation}\label{EstApriori}
\begin{split}
&\cE^2(t) +\|\partial_t u(t)\|_{H^1}^2 + \|\partial_t\theta(t)\|_{L^2}^2 +\|\nabla q(t)\|_{L^2}^2 + \int_0^t (\|\nabla u(s)\|_{H^2}^2+ \|\partial_t u(s)\|_{H^1}^2 + \|\partial_t^2 u(s)\|_{L^2}^2) ds \\
&\leq C_0 \Big( \varepsilon^{-1} \cE^2(0)+\varepsilon  \int_0^t\cE^2(s) ds+ \varepsilon^{-5} \int_0^t\|(\theta,u)(s)\|_{L^2}^2 ds+ \varepsilon^{-1} \int_0^t \cE^3(s)ds\Big),
\end{split}
\end{equation}
where $C_0$ is a generic constant being independent of $\varepsilon$. 
\end{proposition}

 We list below some classical Sobolev estimates frequently used later (see e.g. \cite{AF05}), which are 
\begin{equation}\label{EstClassical}
\begin{split}
\|v\|_{L^4} &\lesssim \|v\|_{L^2}^{1/4}\|v\|_{H^1}^{1/4} \lesssim \|v\|_{H^1},\\
\|v\|_{L^\infty} &\lesssim \|v\|_{H^2}, \\
\| v\|_{H^j} &\lesssim \|v\|_{L^2}^{1/(j+1)} \|v\|_{H^{j+1}}^{j/(j+1)} \lesssim \nu^{-j} \|v\|_{L^2}+ \nu\|v\|_{H^{j+1}} \quad\text{for any }j\geq 0, \nu >0.
\end{split}
\end{equation}
Note that from the continuity equation $\eqref{EqNSK_Pertur}_1$ and the incompressibility condition, we have for any $t\in (0,T^{\max})$ and any $x\in\Omega$ that
\begin{equation}
0 <\frac12 \min_I \rho_0(x_3) < \rho_0(x_3) + \theta(t,x) < \frac32 \max_I \rho_0(x_3).
\end{equation}

Let us start with the two following lemmas. 
\begin{lemma}\label{LemBoundU}
There holds
\begin{equation}\label{BoundDtUinH^1byE}
\|\partial_t u\|_{L^2} \lesssim \|\theta\|_{H^2}+\|u\|_{H^2}, \quad \| \partial_t u\|_{H^1} \lesssim \|\theta\|_{H^3}+ \|u\|_{H^3}.
\end{equation}
\end{lemma}
\begin{proof}
We rewrite  $\eqref{EqNSK_Pertur}_2$ as
\begin{equation}\label{EqNSK_Pertur_2nd}
(\rho_0+\theta)\partial_t u+ \nabla p - \mu\Delta u +(\rho_0+\theta)u\cdot\nabla u +\sigma \nabla(\rho_0+\theta)\Delta \theta + \sigma \rho_0''\nabla\theta +g\theta e_3=0.
\end{equation} It follows from \eqref{EqNSK_Pertur_2nd} and the integration by parts that 
\[
\begin{split}
\int(\rho_0+\theta)|\partial_t u|^2 &= \mu\int \Delta u\cdot\partial_t u -\int (\rho_0+\theta)(u\cdot\nabla u)\cdot\partial_t u \\
&\qquad- \sigma\int \Delta\theta (\nabla(\rho_0+\theta)\cdot\partial_t u) -\sigma \int\rho_0''\nabla\theta \cdot\partial_tu - g\int\theta\partial_tu_3\\
&\lesssim (\|\Delta u\|_{L^2}+ \|(\rho_0+\theta) u\cdot\nabla u\|_{L^2} + \|\Delta \theta \nabla\theta\|_{L^2} + \| \theta\|_{H^2}) \|\partial_t u\|_{L^2}. 
\end{split}
\]
Thanks to Sobolev embedding, $\eqref{EstClassical}_1$ and Young's inequality, we obtain for any $\nu>0$ that, 
\[
\begin{split}
\frac12 \inf_{\overline\Omega} \rho_0 \|\partial_t u\|_{L^2}^2 &\lesssim  (\|\Delta u\|_{L^2} +(1+\|\theta\|_{H^2})\|u\|_{H^2}\|\nabla u\|_{L^2}+ \|\Delta\theta\|_{L^4}\|\nabla\theta\|_{L^4}+ \|\theta\|_{H^2}) \|\partial_t u\|_{L^2} \\
&\lesssim (\|u\|_{H^2}+ \|\theta\|_{H^2}) \|\partial_t u\|_{L^2} \\
&\lesssim \nu \|\partial_t u\|_{L^2}^2+\nu^{-1}  (\|u\|_{H^2} +\|\theta\|_{H^2})^2.
\end{split}
\]
Let $\nu$ be sufficiently small, we obtain $ \|\partial_t u\|_{L^2} \lesssim \|\theta\|_{H^3}+\|u\|_{H^2}$.

Let $j=1,2$ or 3, we have 
\begin{equation}
\begin{split}
&(\rho_0+\theta) \partial_t \partial_ju+ \partial_j(\rho_0+\theta) \partial_t u +\nabla\partial_j q-\mu\Delta\partial_j u + \partial_j ((\rho_0+\theta)u\cdot \nabla u) \\
&\qquad+\sigma \partial_j(\nabla\theta \Delta\theta) + \sigma \partial_j(\rho_0''\nabla\theta+ \rho_0'\Delta \theta e_3) + g\partial_j\theta e_3=0.
\end{split}
\end{equation}
Note that, by Sobolev embedding and $\eqref{EstClassical}_1$,  
\[
\begin{split}
\|\nabla(\nabla \theta\Delta\theta)\|_{L^2} \lesssim \|\nabla^2 \theta\|_{L^4}\|\Delta\theta\|_{L^4} + \|\nabla\theta\|_{H^2}\|\Delta\theta\|_{H^1} \lesssim \|\theta\|_{H^3}^2.
\end{split}
\]
Hence, by the same arguments as the proof of $ \|\partial_t u\|_{L^2} \lesssim \|\theta\|_{H^3}+\|u\|_{H^2}$, we obtain 
\[
\begin{split}
\int(\rho_0+\theta)|\partial_t \partial_j u|^2 &= -\int \partial_j(\rho_0+\theta) \partial_t u\cdot\partial_t\partial_j u  -\int \partial_j ((\rho_0+\theta)(u\cdot\nabla u))\cdot\partial_t\partial_j u \\
&\qquad + \mu\int \Delta\partial_j u\cdot\partial_t \partial_j u -\sigma \int \partial_j(\nabla\theta \Delta\theta  +\rho_0''\nabla\theta+ \rho_0'\Delta \theta e_3) \cdot\partial_t \partial_j u \\
&\qquad- g\int\partial_j \theta\partial_t \partial_j u_3 \\
&\lesssim  ((1+\|\theta\|_{H^3} )(\|\partial_t u\|_{L^2} + \|u\|_{H^3}^2+\|\theta\|_{H^3}) + \|u\|_{H^3}) \| \partial_t\partial_j u\|_{L^2}\\
&\lesssim ( \|u\|_{H^3}+\|\theta\|_{H^3})\| \partial_t\partial_j u\|_{L^2}.
\end{split}
\]
By Young's inequality, we obtain $ \|\partial_j \partial_t u\|_{L^2} \lesssim \|\theta\|_{H^3}+\|u\|_{H^3}$. The inequality \eqref{BoundDtUinH^1byE}, i.e. Lemma \ref{LemBoundU} thus follows. 
\end{proof}

\begin{lemma}\label{LemBoundTheta}
There holds 
\begin{equation}\label{BoundDtThetaH2}
\|\partial_t\theta\|_{H^2} \lesssim \|u\|_{H^2}, \quad \|\partial_t \theta\|_{H^3}\lesssim \|u\|_{H^3}.
\end{equation}
\end{lemma}
\begin{proof}
From $\eqref{EqNSK_Pertur}_1$ and Sobolev embedding, we obtain
\begin{equation}
\|\partial_t \theta\|_{H^2} \lesssim \|u_3\|_{H^2}+ \|u\cdot\nabla\theta\|_{H^2}\lesssim   \|u \|_{H^2}(1+ \|\nabla\theta\|_{H^2})\lesssim \|u\|_{H^2}.
\end{equation}
Similarly, one has
\begin{equation}
\|\partial_t \theta\|_{H^3} \lesssim \|u_3\|_{H^3}+ \|u\cdot\nabla\theta\|_{H^3}\lesssim \|u_3\|_{H^3}+   \|u \|_{H^2}  \|\nabla \theta\|_{H^3}+ \|u\|_{H^3}\|\nabla\theta\|_{H^2}\lesssim \|u\|_{H^3}.
\end{equation}
Lemma \ref{LemBoundTheta} is proven. 
\end{proof}

We now derive \textit{a priori} energy estimates for the density and velocity in Propositions \ref{PropNormU_DtU_L2}, \ref{PropNormU_DtU_H1}, \ref{PropNormU_H3_Theta_H4}.

\begin{proposition}\label{PropNormU_DtU_L2}
The following inequalities hold
\begin{equation}\label{EstNormU_L2_Theta_H1}
\begin{split}
\|u(t)\|_{L^2}^2+ \|\nabla \theta(t)\|_{L^2}^2 + \int_0^t \|\nabla u(s)\|_{L^2}^2 &\lesssim \cE^2(0) + \varepsilon \int_0^t \cE^2(s)ds+ \varepsilon^{-2}\int_0^t \|(\theta,u)(s)\|_{L^2}^2ds \\
&\qquad +\int_0^t \cE^3(s)ds.
\end{split}
\end{equation}
\begin{equation}\label{EstNormDtU_L2_DtTheta_H1}
\begin{split}
\|\partial_t u(t)\|_{L^2}^2 +\|\nabla\partial_t\theta(t)\|_{L^2}^2 + \int_0^t \|\nabla \partial_t u(s)\|_{L^2}^2 ds &\lesssim \cE^2(0) + \varepsilon \int_0^t \cE^2(s)ds \\
&\qquad+\varepsilon^{-2} \int_0^t  \|(\theta,u)(s)\|_{L^2}^2ds +\int_0^t \cE^3(s)ds.
\end{split}
\end{equation}
\end{proposition}
\begin{proof}
Let us compute that 
\begin{equation}\label{DtU_L2}
\frac12 \frac{d}{dt} \int (\rho_0+\theta)|u|^2 = \int (\rho_0+\theta) \partial_t u\cdot u+\frac12 \int \partial_t \theta |u|^2. 
\end{equation}
By the integration by parts, 
\begin{equation}\label{IntDtThetaU^2}
\int \partial_t \theta|u|^2 = - \int u\cdot \nabla(\rho_0+\theta) |u|^2 = \int (\rho_0+\theta) u \cdot\nabla |u|^2.
\end{equation}
Substituting \eqref{EqNSK_Pertur_2nd} and \eqref{IntDtThetaU^2} into \eqref{DtU_L2}, we obtain 
\begin{equation}\label{DtU_L2_1}
\begin{split}
\frac12 \frac{d}{dt} \int (\rho_0+\theta)|u|^2 &= -\int (\rho_0+\theta) (u\cdot \nabla u)\cdot u - \int (\nabla p-\mu\Delta u)\cdot u -\sigma \int \rho_0'' \nabla \theta \cdot u\\
&\qquad  -\sigma \int u\cdot \nabla(\rho_0+\theta) \Delta \theta -g\int \theta u_3 +\frac12\int\partial_t \theta|u|^2 \\
&= - \frac12 \int (\rho_0+\theta) (u\cdot \nabla u)\cdot u - \mu \int |\nabla u|^2 -\sigma \int \rho_0'' \nabla \theta \cdot u\\
&\qquad  -\sigma \int u\cdot \nabla(\rho_0+\theta) \Delta \theta -g\int \theta u_3
\end{split}
\end{equation}
Note also that, due to $\eqref{EqNSK_Pertur}_1$,  
\[
-\sigma \int u\cdot \nabla(\rho_0+\theta) \Delta \theta = \sigma \int \partial_t \theta \Delta \theta = -\frac{\sigma}2 \frac{d}{dt} \int|\nabla \theta|^2.
\]
Hence, it follows from \eqref{DtU_L2_1} that 
\begin{equation}\label{DtU_L2_2}
\begin{split}
\frac12 \frac{d}{dt} \int( (\rho_0+\theta)|u|^2 +\sigma|\nabla \theta|^2) + \mu\int|\nabla u|^2& =  -\sigma \int\rho_0'' \nabla \theta\cdot u -g\int \theta u_3.
\end{split}
\end{equation}
We estimate the r.h.s of \eqref{DtU_L2_2}.
Using the interpolation inequality $\eqref{EstClassical}_3$ and Young's inequality yields
\[
\int \rho_0'' \nabla\theta\cdot u \lesssim \|\theta\|_{H^1}\|u\|_{L^2} \lesssim (\varepsilon \|\theta\|_{H^2} +\varepsilon^{-2}\|\theta\|_{L^2}) \|u\|_{L^2} \lesssim \varepsilon \cE^2 + \varepsilon^{-1} \|(\theta, u)\|_{L^2}^2.
\]
That implies
\[
\frac{d}{dt} \int( (\rho_0+\theta)|u|^2 +|\nabla \theta|^2) + \int|\nabla u|^2  \lesssim \varepsilon \cE^2 + \varepsilon^{-2} \|(\theta, u)\|_{L^2}^2 +\cE^3.
\]
Integrating the resulting inequality in time from $0$ to $t$ and noticing that $\inf_\Omega (\rho_0+\theta) >0$, we deduce \eqref{EstNormU_L2_Theta_H1}.

We now prove \eqref{EstNormDtU_L2_DtTheta_H1}. Let us take the derivative in time to \eqref{EqNSK_Pertur_2nd} to get 
\begin{equation}\label{EqPerturDt^2U}
\begin{split}
&(\rho_0+\theta)\partial_t^2 u+ \partial_t\theta \partial_t u + (\rho_0+\theta)( \partial_t u\cdot\nabla u +u\cdot\nabla\partial_t u) + \partial_t\theta u\cdot\nabla u \\
&\quad+ \nabla \partial_t q - \mu\Delta \partial_t u +\sigma \nabla(\rho_0+\theta)\Delta \partial_t\theta + \sigma \nabla\partial_t\theta \Delta\theta + \sigma \rho_0''\nabla\partial_t\theta +g\partial_t\theta e_3=0.
\end{split}
\end{equation}
Multiplying both sides of \eqref{EqPerturDt^2U} by $\partial_t u$ and integrating over $\Omega$, one has
\begin{equation}
\begin{split}
&\int   (\rho_0+\theta)\partial_t^2 u \cdot \partial_t u  +\int(\nabla \partial_t q-\mu\Delta \partial_t u) \cdot\partial_t u \\
&= -\int \partial_t\theta |\partial_t u|^2-\int (\rho_0+\theta)( \partial_t u\cdot\nabla u +u\cdot\nabla\partial_t u) \cdot\partial_t u -\int \partial_t\theta (u\cdot\nabla u)\cdot\partial_t u \\
&\qquad-\sigma \int (\nabla(\rho_0+\theta)\Delta\partial_t\theta) \cdot \partial_t u -\sigma \int \Delta\theta \nabla\partial_t\theta\cdot\partial_t u-\sigma \int \rho_0''\nabla\partial_t\theta \cdot\partial_t u  -g\int\partial_t\theta \partial_t u_3.
\end{split}
\end{equation}
That is equivalent to 
\begin{equation}\label{1_EqDtIntDtU_L2}
\begin{split}
& \frac12 \frac{d}{dt}\int (\rho_0+\theta)| \partial_t u|^2  +\int(\nabla \partial_t q-\mu\Delta \partial_t u) \cdot\partial_t u \\
&= \frac12 \int \partial_t\theta |\partial_t u|^2-\int (\rho_0+\theta)( \partial_t u\cdot\nabla u +u\cdot\nabla\partial_t u) \cdot\partial_t u -\int \partial_t\theta (u\cdot\nabla u)\cdot\partial_t u \\
&\qquad-\sigma \int (\nabla(\rho_0+\theta)\Delta\partial_t\theta) \cdot \partial_t u -\sigma \int \Delta\theta \nabla\partial_t\theta\cdot\partial_t u-\sigma \int \rho_0''\nabla\partial_t\theta \cdot\partial_t u  -g\int\partial_t\theta \partial_t u_3.
\end{split}
\end{equation}
Note that 
\begin{equation}\label{Int2ndDeltaDtTheta}
 \begin{split}
- \int (\nabla(\rho_0+\theta)\Delta\partial_t\theta) \cdot \partial_t u  &= \int \Delta\partial_t\theta ( \partial_t^2\theta + u\cdot\nabla\partial_t\theta) \\
& = - \frac12 \frac{d}{dt} \int|\nabla\partial_t\theta|^2  + \int \Delta\partial_t\theta (u\cdot\nabla\partial_t\theta).
\end{split}
\end{equation}
By the integration by parts, 
\begin{equation}\label{IntDtThetaD_tU^2}
\int \partial_t \theta|\partial_t u|^2 = - \int u\cdot \nabla(\rho_0+\theta) |\partial_t u|^2 = \int (\rho_0+\theta) u \cdot\nabla |\partial_t u|^2.
\end{equation}
Substituting \eqref{Int2ndDeltaDtTheta} and \eqref{IntDtThetaD_tU^2} into \eqref{1_EqDtIntDtU_L2},  it yields
\begin{equation}\label{2_EqDtIntDtU_L2}
\begin{split}
& \frac12 \frac{d}{dt}\int ( (\rho_0+\theta)| \partial_t u|^2 +\sigma |\nabla\partial_t\theta|^2) +\frac\mu2 \int |\nabla \partial_t u|^2 \\
&=-\int (\rho_0+\theta)( \partial_t u\cdot\nabla u ) \cdot\partial_t u -\int \partial_t\theta (u\cdot\nabla u)\cdot\partial_t u -\sigma \int \Delta\partial_t\theta (u\cdot\nabla\partial_t\theta)\\
&\qquad -\sigma \int \Delta\theta \nabla\partial_t\theta\cdot\partial_t u-\sigma \int \rho_0''\nabla\theta \cdot\partial_t u  -g\int\theta \partial_t u_3.
\end{split}
\end{equation}
Let us estimate the r.h.s of \eqref{2_EqDtIntDtU_L2} by using Sobolev embedding and \eqref{EstClassical}. 
We have
\begin{equation}\label{3_EqDtIntDtU_L2}
\begin{split}
\int (\rho_0+\theta)( \partial_t u\cdot\nabla u ) \cdot\partial_t u &\lesssim \|\rho_0+\theta\|_{H^2}\|\nabla u\|_{H^2} \|\partial_t u\|_{L^2}^2 \\
&\lesssim (1+\|\theta\|_{H^2}) \|u\|_{H^3} \|\partial_t u\|_{L^2}^2,
\end{split}
\end{equation}
and 
\begin{equation}\label{7_EqDtIntDtU_L2}
\sigma \int \rho_0''\nabla\theta \cdot\partial_t u  +g\int\theta \partial_t u_3\lesssim \|\theta\|_{H^1}\|\partial_tu\|_{L^2}.
\end{equation}
Using Lemma \ref{LemBoundTheta} also, we obtain
\begin{equation}\label{4_EqDtIntDtU_L2}
\int \partial_t\theta (u\cdot\nabla u)\cdot\partial_t u \lesssim \|\partial_t\theta\|_{H^2}\|u\|_{H^2}\|\nabla u\|_{L^2}\|\partial_t u\|_{L^2} \lesssim \|u\|_{H^2}^3 \|\partial_t u\|_{L^2},
\end{equation}
and
\begin{equation}\label{5_EqDtIntDtU_L2}
\int \Delta\partial_t\theta (u\cdot\nabla\partial_t\theta) \lesssim \|\Delta \partial_t\theta\|_{L^2} \|u\|_{L^4}\|\nabla\partial_t\theta\|_{L^4} \lesssim  \|\partial_t \theta\|_{H^2}^2 \|u\|_{H^1}  \lesssim \|u\|_{H^2}^3,  
\end{equation}
\begin{equation}\label{6_EqDtIntDtU_L2}
 \int \Delta\theta \nabla\partial_t\theta\cdot\partial_t u \lesssim \|\Delta\theta\|_{L^4}\|\nabla\partial_t\theta\|_{L^4} \|\partial_t u\|_{L^2} \lesssim \|\theta\|_{H^3}\|\partial_t\theta\|_{H^2}\|\partial_tu\|_{L^2}.
\end{equation}
In view of \eqref{BoundDtUinH^1byE} and those above estimates \eqref{3_EqDtIntDtU_L2}, \eqref{7_EqDtIntDtU_L2}, \eqref{4_EqDtIntDtU_L2}, \eqref{5_EqDtIntDtU_L2} and \eqref{6_EqDtIntDtU_L2}, we have 
\begin{equation}
\begin{split}
 \frac{d}{dt}\int ( (\rho_0+\theta)| \partial_t u|^2 + |\nabla\partial_t\theta|^2) +\int |\nabla \partial_t u|^2 &\lesssim \|\theta\|_{H^1} \|\partial_t u\|_{L^2} + \cE^3  \\
&\lesssim  \|\theta\|_{H^2}^2+\|u\|_{H^2}^2 +\cE^3\\
&\lesssim \varepsilon \cE^2 + \varepsilon^{-2} \|(\theta, u)\|_{L^2}^2 +\cE^3.
\end{split}
\end{equation}
Integrating the resulting inequality in time from $0$ to $t$, we deduce
\[
\begin{split}
\int ( (\rho_0+\theta)| \partial_t u|^2 + |\nabla\partial_t\theta|^2)(t) +\mu \int_0^t \|\nabla\partial_tu(s)\|_{L^2}^2 &\lesssim \int ( (\rho_0+\theta)| \partial_t u|^2 + |\nabla\partial_t\theta|^2)(0) \\
&\qquad +\int_0^t ( \cE^2 + \varepsilon^{-2} \|(\theta, u)\|_{L^2}^2 +\cE^3)(s)ds.
\end{split}
\]
This yields
\[\begin{split}
\|\partial_tu(t)\|_{L^2}^2 +\|\nabla\partial_t\theta(t)\|_{L^2}^2 + \int_0^t \|\nabla\partial_tu(s)\|_{L^2}^2  &\lesssim \|\partial_t u(0)\|_{L^2}^2+\|\partial_t\theta(0)\|_{H^1}^2\\
&\qquad+\int_0^t ( \cE^2 + \varepsilon^{-2} \|(\theta, u)\|_{L^2}^2 +\cE^3)(s)ds.
\end{split}\]
Together with \eqref{BoundDtUinH^1byE} and \eqref{BoundDtThetaH2}, we get \eqref{EstNormDtU_L2_DtTheta_H1}.  Proposition \ref{PropNormU_DtU_L2} is proven. 
\end{proof}

\begin{proposition}\label{PropNormU_DtU_H1}
The following inequalities hold
\begin{equation}\label{EstIntDtU_L2}
\|\nabla u(t)\|_{L^2}^2 + \int_0^t \|\partial_t u(s)\|_{L^2}^2ds \lesssim \cE^2(0)+ \int_0^t (\varepsilon \|\theta(s)\|_{H^3}^2 + \varepsilon^{-2}\|\theta(s)\|_{L^2}^2 +\cE^4(s))ds,
\end{equation}
\begin{equation}\label{EstIntDt^2U_L2}
\|\nabla \partial_tu(t)\|_{L^2}^2 + \int_0^t \|\partial_t^2 u(s)\|_{L^2}^2ds \lesssim  \cE^2(0)+ \int_0^t (\varepsilon \|u(s)\|_{H^3}^2 + \varepsilon^{-2}\|u(s)\|_{L^2}^2 +\cE^4(s))ds
\end{equation}
\end{proposition}
\begin{proof}

Let us prove \eqref{EstIntDtU_L2} first.  We multiply both sides of $\eqref{EqNSK_Pertur}_2$ by $\partial_t u$ and integrate to have that
\begin{equation}\label{EqD^alphaDtNablaU}
\begin{split}
&\int   (\rho_0+\theta)\partial_t u \cdot \partial_tu  +\int(\nabla q-\mu\Delta u) \cdot\partial_t u \\
&= -\int  (\rho_0+\theta)(u\cdot\nabla u) \cdot \partial_t u -\sigma \int (\nabla(\rho_0+\theta)\Delta\theta) \cdot \partial_t u -\sigma \int \rho_0''\nabla\theta \cdot\partial_t u  -g\int\theta \partial_t u_3.
\end{split}
\end{equation}
Hence, using the integration by parts, 
\begin{equation}\label{EqDtUIntNablaU}
\begin{split}
\int (\rho_0+\theta)|\partial_tu|^2  + \frac\mu2 \frac{d}{dt} \int |\nabla  u|^2   &= -\int  ((\rho_0+\theta)u\cdot\nabla u) \cdot \partial_t u -\sigma \int (\nabla(\rho_0+\theta) \cdot \partial_t u)\Delta\theta\\
&\qquad -\sigma \int \rho_0''\nabla\theta\cdot\partial_t u  -g\int \theta\partial_t  u_3.
\end{split}
\end{equation}
We bound each integral in the r.h.s of  \eqref{EqDtUIntNablaU}. Keep using Sobolev embedding, we have
\begin{equation}\label{1_EstIntDtU_L2}
\begin{split}
\int  ((\rho_0+\theta)u\cdot\nabla u) \cdot \partial_t u  &\lesssim \| (\rho_0+\theta)u\cdot \nabla u\|_{L^2}\|\partial_t u\|_{L^2}  \\
&\lesssim (1+\|\theta\|_{H^2})\|u\|_{H^2}^2 \|\partial_t u\|_{L^2}.
\end{split}
\end{equation}
For the second integral, we observe 
\begin{equation}\label{2_EstIntDtU_L2}
\begin{split}
\int (\nabla(\rho_0+\theta) \cdot \partial_t u)\Delta\theta &\lesssim \|\nabla(\rho_0+\theta)\|_{H^2}\|\Delta\theta\|_{L^2} \|\partial_t u\|_{L^2} \\
&\lesssim (1+\|\theta\|_{H^3})\|\theta\|_{H^2} \|\partial_t u\|_{L^2}\\
\end{split}
\end{equation}
We also have
\begin{equation}\label{3_EstIntDtU_L2}
\begin{split}
\int \rho_0''\nabla\theta \cdot\partial_t u +\int\theta \partial_t u_3 &\lesssim \|\theta\|_{H^1}\|\partial_t u\|_{L^2}.
\end{split}
\end{equation}
For any $\nu>0$, it follows from \eqref{1_EstIntDtU_L2}, \eqref{2_EstIntDtU_L2}, \eqref{3_EstIntDtU_L2}   and Young's inequality that
\begin{equation}\label{EqDtNablaU}
\begin{split}
\frac12 \min_I \rho_0 \int |\partial_tu|^2  + \frac\mu2 \frac{d}{dt} \int |\nabla  u|^2   &\lesssim \nu \|\partial_t u\|_{L^2}^2 + \nu^{-1}(\|\theta\|_{H^2}^2 +\cE^4).
\end{split}
\end{equation}
We choose  sufficiently small $\nu$ and use $\eqref{EstClassical}_3$ to obtain 
\[
\|\partial_tu\|_{L^2}^2  + \frac{d}{dt}\|\nabla  u\|_{L^2}^2   \lesssim \|\theta\|_{H^2}^2 +\cE^4 \lesssim \varepsilon \|\theta\|_{H^3}^2+ \varepsilon^{-1}\|\theta\|_{L^2}^2 + \cE^4.
\]
Integrating in time from 0 to $t$, the inequality \eqref{EstIntDtU_L2} follows.

Now we prove \eqref{EstIntDt^2U_L2}. Multiplying by $\partial_t^2 u$ to both sides of \eqref{EqPerturDt^2U} and integrating over $\Omega$ by parts, it yields
\begin{equation}\label{EstIntDt^2U_L2_1}
\begin{split}
&\int (\rho_0+\theta)|\partial_t^2 u|^2 + \frac{\mu}2 \frac{d}{dt}\int |\nabla\partial_t u|^2\\
&= -\int \partial_t\theta \partial_t u \cdot\partial_t^2 u +\int (\rho_0+\theta) (\partial_t u\cdot\nabla u+ u\cdot\nabla \partial_tu)\partial_t^2 u - \int \partial_t \theta(u\cdot\nabla u)\cdot\partial_t ^2 u \\
&\qquad - \sigma \int \Delta\partial_t\theta  \nabla(\rho_0+\theta)\cdot \partial_t^2 u - \sigma \int \Delta\theta \nabla\partial_t\theta \cdot\partial_t^2 u -\sigma \int\rho_0'' \nabla\partial_t\theta \cdot\partial_t^2 u -g\int \partial_t\theta\partial_t^2 u_3.
\end{split}
\end{equation}
We now estimate each integral in the r.h.s of \eqref{EstIntDt^2U_L2_1} by using the interpolation inequality. Thanks to \eqref{BoundDtUinH^1byE} and \eqref{BoundDtThetaH2}, we have that
\begin{equation}\label{1_EstIntDt^2U_L2_1}
\begin{split}
\int \partial_t\theta \partial_t u \cdot\partial_t^2 u 
\lesssim \|\partial_t\theta\|_{H^2} \|\partial_t u\|_{L^2} \|\partial_t^2 u\|_{L^2} \lesssim \|u\|_{H^2}\|(\theta,u)\|_{H^2} \|\partial_t^2 u\|_{L^2},
\end{split}
\end{equation}
that
\begin{equation}\label{2_EstIntDt^2U_L2_1}
\begin{split}
\int (\rho_0+\theta) (\partial_t u\cdot\nabla u+ u\cdot\nabla \partial_tu)\partial_t^2 u &\lesssim \|\rho_0+\theta\|_{H^2} (\|\partial_t u\|_{L^2} \|\nabla u\|_{H^2}+ \|u\|_{H^2}\|\nabla \partial_t u\|_{L^2}) \|\partial_t^2 u\|_{L^2} \\
&\lesssim (1+\|\theta\|_{H^2}) \|(\theta, u)\|_{H^3} \|u\|_{H^3}\|\partial_t^2 u\|_{L^2},
\end{split} 
\end{equation}
that
\begin{equation}\label{3_EstIntDt^2U_L2_1}
\int \partial_t \theta(u\cdot\nabla u)\cdot\partial_t ^2 u \lesssim \|\partial_t\theta\|_{H^2} \|u\|_{H^2}\|\nabla u\|_{L^2} \|\partial_t^2 u\|_{L^2} \lesssim \|u\|_{H^2}^3 \|\partial_t^2 u\|_{L^2},
\end{equation}
that
\begin{equation}\label{4_EstIntDt^2U_L2_1}
 \int \Delta\partial_t\theta  \nabla(\rho_0+\theta)\cdot \partial_t^2 u \lesssim \|\partial_t\theta\|_{H^2} \|\rho_0+\theta\|_{H^1}\|\partial_t^2 u\|_{L^2} \lesssim \|u\|_{H^2}(1+\|\theta\|_{H^1})\|\partial_t^2 u\|_{L^2},
\end{equation}
and that
\begin{equation}\label{5_EstIntDt^2U_L2_1}
\sigma \int\rho_0'' \nabla\partial_t\theta \cdot\partial_t^2 u +g\int \partial_t\theta\partial_t^2 u_3 \lesssim \|\partial_t\theta\|_{H^1} \|\partial_t^2 u\|_{L^2} \lesssim \|u\|_{H^2} \|\partial_t^2 u\|_{L^2}.
\end{equation}
Using $\eqref{EstClassical}_1$ also, 
\begin{equation}\label{6_EstIntDt^2U_L2_1}
\int \Delta\theta \nabla\partial_t\theta \cdot\partial_t^2 u \lesssim \|\Delta \theta\|_{L^4} \|\partial_t\theta\|_{L^4}\|\partial_t^2 u\|_{L^2} \lesssim \|\theta\|_{H^3} \|u\|_{H^2} \|\partial_t^2 u\|_{L^2}.
\end{equation}
Combining those above estimates \eqref{1_EstIntDt^2U_L2_1},  \eqref{2_EstIntDt^2U_L2_1},  \eqref{3_EstIntDt^2U_L2_1},  \eqref{4_EstIntDt^2U_L2_1},  \eqref{5_EstIntDt^2U_L2_1} and  \eqref{6_EstIntDt^2U_L2_1}, we get
\begin{equation}
\begin{split}
\int (\rho_0+\theta)|\partial_t^2 u|^2 + \frac{\mu}2 \frac{d}{dt}\int |\nabla\partial_t u|^2 &\lesssim (\|u\|_{H^2}+\cE^2) \|\partial_t^2 u\|_{L^2} \\
&\lesssim \nu \|\partial_t^2 u\|_{L^2}^2 +\nu^{-1}( \|u\|_{H^2}^2 +\cE^4).
\end{split}
\end{equation}
We choose $\nu>0$  sufficiently small and use $\eqref{EstClassical}_3$ to obtain 
\[
\|\partial_t^2 u\|_{L^2}^2 + \frac{d}{dt} \|\nabla\partial_t u\|_{L^2}^2  \lesssim \|u\|_{H^2}^2 +\cE^4 \lesssim \varepsilon\| u\|_{H^3}^2 +\varepsilon^{-2} \|u\|_{L^2}^2 + \cE^4.
\]
Integrating in time from 0 to $t$ the resulting inequality, we obtain. 
\begin{equation}\label{7_EstIntDt^2U_L2_1}
\begin{split}
\|\nabla\partial_t u(t)\|_{L^2}^2 +\int_0^t \|\partial_t^2 u(s)\|_{L^2}^2 ds &\lesssim \|\nabla\partial_t u(0)\|_{L^2}^2+   \int_0^t (\varepsilon\| u(s)\|_{H^3}^2 +\varepsilon^{-2} \|u(s)\|_{L^2}^2+ \cE^4(s))ds.
\end{split}
\end{equation}
The inequality \eqref{EstIntDt^2U_L2} thus follows from  \eqref{7_EstIntDt^2U_L2_1} and \eqref{BoundDtUinH^1byE}. The proof of Proposition \ref{PropNormU_DtU_H1} is complete.
\end{proof}

Let us continue with $H^2$-norm of the velocity. 

\begin{proposition}\label{PropNormU_H3_Theta_H4}
There holds 
\begin{equation}\label{EstNormU_H3_Theta_H4}
\begin{split}
\|u(t)\|_{H^2}^2 +\|\nabla\theta(t)\|_{H^2}^2  + \int_0^t \|\nabla u(s)\|_{H^2}^2 ds  &\lesssim \cE^2(0)+ \varepsilon \int_0^t \cE^2(s) ds + \varepsilon^{-2}\int_0^t \|(\theta,u)(s)\|_{L^2}^2 ds \\
&\qquad+ \int_0^t\cE^3(s)ds.
\end{split}
\end{equation}
\end{proposition}
\begin{proof}
For  $\alpha \in \N^3$ such that  $ |\alpha|=2$, applying the operator $\partial^\alpha$ to both sides of \eqref{EqNSK_Pertur_2nd} and multiplying the resulting equation by $\partial^\alpha u$, we obtain 
\begin{equation}\label{EqD^alphaU}
\begin{split}
&\int\partial^{\alpha}((\rho_0+\theta) \partial_t u) \cdot \partial^{\alpha}u +\sigma \int\partial^{\alpha}(\nabla(\rho_0+\theta) \Delta\theta) \cdot\partial^{\alpha}u  - \mu\int \Delta\partial^\alpha u \cdot \partial^{\alpha} u  \\
&= -\int\partial^{\alpha}((\rho_0+\theta)u\cdot\nabla u) \cdot \partial^{\alpha} u  -\sigma \int\partial^{\alpha}(\rho_0'' \nabla\theta)\cdot \partial^{\alpha}u - g\int\partial^{\alpha}\theta \partial^{\alpha}u_3.
\end{split}
\end{equation}
after using the integration by parts. Using $\eqref{EqNSK_Pertur}_1$, we compute  that 
\begin{equation}\label{EqD^alphaU_1}
\begin{split}
\int\partial^{\alpha}((\rho_0+\theta) \partial_t u) \cdot \partial^{\alpha}u &= \int(\rho_0+\theta) \partial_t\partial^\alpha u\cdot \partial^\alpha u + \sum_{0\neq \beta \leq \alpha} \int \partial^\beta(\rho_0+\theta) \partial^{\alpha-\beta}\partial_t u\cdot\partial^\alpha u\\
&=\frac12 \frac{d}{dt} \int(\rho_0+\theta) |\partial^\alpha u|^2 + \frac12 \int (u\cdot\nabla(\rho_0+\theta))|\partial^\alpha u|^2 \\
&\qquad + \sum_{0\neq \beta \leq \alpha} \int \partial^\beta(\rho_0+\theta) \partial^{\alpha-\beta}\partial_t u\cdot\partial^\alpha u.
\end{split}
\end{equation}
Furthermore, from $\eqref{EqNSK_Pertur}_1$, one has that
\[
\partial_t\partial^{\alpha}\theta +\partial^{\alpha}u\cdot \nabla (\rho_0+\theta)=- \sum_{\beta\in \N^3, 0\neq \beta \leq \alpha} \partial^{\alpha-\beta}u \cdot\nabla \partial^\beta(\rho_0+\theta).
\]
This yields,
\begin{equation}\label{EqD^alphaU_2}
\begin{split}
\int \partial^{\alpha}\Delta\theta (\nabla(\rho_0+\theta) \cdot \partial^{\alpha}u)   &= -\int \partial_t\partial^{\alpha}\theta \partial^{\alpha}\Delta\theta  -   \sum_{ 0\neq \beta\leq \alpha} \int\partial^{\alpha-\beta}u \cdot\nabla \partial^\beta(\rho_0+\theta) \Delta\partial^{\alpha}\theta  \\
&= \frac12 \frac{d}{dt} \int|\partial^{\alpha}\nabla\theta|^2  -   \sum_{0\neq \beta\leq \alpha} \int\partial^{\alpha-\beta}u \cdot\nabla \partial^\beta(\rho_0+\theta) \Delta\partial^{\alpha}\theta.
 \end{split}
\end{equation}
Combining \eqref{EqD^alphaU_1} and \eqref{EqD^alphaU_2}, we  rewrite the l.h.s of \eqref{EqD^alphaU} as
\[
\begin{split}
&\int\partial^{\alpha}((\rho_0+\theta) \partial_t u) \cdot \partial^{\alpha}u +\sigma \int\partial^{\alpha}(\nabla(\rho_0+\theta) \Delta\theta) \cdot\partial^{\alpha}u  - \mu\int \Delta\partial^\alpha u \cdot \partial^{\alpha} u  \\
&= \frac12 \frac{d}{dt} \int( (\rho_0+\theta) |\partial^{\alpha}u|^2 +\sigma |\partial^{\alpha}\nabla\theta|^2) + \mu \int |\nabla\partial^\alpha u|^2  \\
&\qquad+ \frac12 \int (u\cdot\nabla(\rho_0+\theta))|\partial^\alpha u|^2 + \sum_{0\neq \beta \leq \alpha} \int \partial^\beta(\rho_0+\theta) \partial^{\alpha-\beta}\partial_t u\cdot\partial^\alpha u\\
&\qquad + \sum_{ 0\neq \beta\leq \alpha} \int\partial^{\alpha-\beta}u \cdot\nabla \partial^\beta(\rho_0+\theta) \Delta\partial^{\alpha}\theta.
\end{split}
\]
Thanks to Sobolev embedding, it can be seen that
\begin{equation}\label{EqD^alphaULeft_1}
\begin{split}
\int u\cdot\nabla(\rho_0+\theta)|\partial^\alpha u|^2 &\gtrsim - \|u\|_{H^2}\|\rho_0+\theta\|_{H^3}\|\partial^\alpha u\|_{L^2}^2 \\
& \gtrsim -(1+\|\theta\|_{H^3}) \|u\|_{H^2}^3.
\end{split}
\end{equation}
Using $\eqref{EstClassical}_1$  and Cauchy-Schwarz's inequality also, we get 
\begin{equation}\label{EqD^alphaULeft_2}
\begin{split}
&\sum_{0\neq \beta \leq \alpha} \int \partial^\beta(\rho_0+\theta) \partial^{\alpha-\beta}\partial_t u\cdot\partial^\alpha u  \\
&\gtrsim -( \|\nabla (\rho_0+\theta)\|_{H^2} \|\nabla \partial_t u\|_{L^2}+ \|\nabla^2(\rho_0+\theta)\|_{L^4} \|\partial_t u\|_{L^4}) \|\nabla^2 u\|_{L^2} \\
&\gtrsim - (1+\|\theta\|_{H^3}) \|\partial_t u\|_{H^1} \|u\|_{H^2}.
\end{split}
\end{equation}
We have that
\[
\begin{split}
& \sum_{0\neq \beta\leq \alpha} \int\partial^{\alpha-\beta}u \cdot\nabla \partial^\beta(\rho_0+\theta) \Delta\partial^{\alpha}\theta \\
&= \int u\cdot \nabla \partial^\alpha(\rho_0+\theta) \Delta\partial^\alpha\theta +\sum_{|\beta| =1} \int\partial^{\alpha-\beta}u \cdot\nabla \partial^\beta(\rho_0+\theta) \Delta\partial^{\alpha}\theta.
\end{split}
\]
Using Einstein's convention and the integration by parts, we obtain
\[
\begin{split}
\int u\cdot \nabla \partial^\alpha(\rho_0+\theta) \Delta\partial^\alpha\theta &= - \int \partial_j (u_i \partial_i \partial^\alpha(\rho_0+\theta) ) \partial_j \partial^\alpha\theta \\
&= -\int [ \partial_j(u_i \partial_i \partial^\alpha\rho_0) +\partial_j u_i \partial_i\partial^\alpha \theta] \partial_j \partial^\alpha\theta -\frac12\int u_i \partial_i |\partial_j\partial^\alpha \theta|^2 \\
&=-\int [\partial_j(u_i \partial_i \partial^\alpha\rho_0) + \partial_j u_i \partial_i\partial^\alpha \theta] \partial_j \partial^\alpha\theta + \frac12 \int |\nabla\partial^\alpha \theta|^2 \text{div}u,
\end{split}
\]
It yields
\[
\int u\cdot \nabla \partial^\alpha(\rho_0+\theta) \Delta\partial^\alpha\theta=-\int [\partial_j(u_i \partial_i \partial^\alpha\rho_0) + \partial_j u_i \partial_i\partial^\alpha \theta] \partial_j \partial^\alpha\theta. 
\]
Thanks to Cauchy-Schwarz's inequality and Sobolev embedding, we estimate that
\begin{equation}\label{EqD^alphaULeft_3}
\begin{split}
\int u\cdot \nabla \partial^\alpha(\rho_0+\theta) \Delta\partial^\alpha\theta  &\gtrsim -\|\nabla u\|_{L^2} \|\nabla\partial^\alpha \theta\|_{L^2} - \|\nabla u\|_{H^2} \|\nabla\partial^\alpha\theta\|_{L^2}^2 \\
&\gtrsim -\|u\|_{H^1}\|\theta\|_{H^3} - \|u\|_{H^3}\|\theta\|_{H^3}^2.
\end{split}
\end{equation}
In a same way, we have
\[
\begin{split}
&\sum_{ |\beta| =1} \int\partial^{\alpha-\beta}u \cdot\nabla \partial^\beta(\rho_0+\theta) \Delta\partial^{\alpha}\theta \\
&= -\sum_{|\beta| =1} \int \partial_j [\partial^{\alpha-\beta}u_i \partial_i \partial^\beta\rho_0]\partial_j \partial^\alpha\theta -\sum_{ |\beta| =1} \int \partial_j [\partial^{\alpha-\beta}u_i \partial_i \partial^\beta\theta]\partial_j \partial^\alpha\theta\\
&\gtrsim -\sum_{ |\beta| =1}\|\partial^{\alpha-\beta} u\|_{H^1}\|\nabla\partial^\alpha\theta\|_{L^2} - \sum_{ |\beta| =1} \|\partial^{\alpha-\beta} u\|_{H^2} \| \partial^\beta \theta\|_{H^2} \|\nabla\partial^\alpha\theta\|_{L^2} \\
&\qquad - \sum_{ |\beta| =1} \|\nabla \partial^{\alpha-\beta} u\|_{L^4}\|\nabla\partial^\beta \theta\|_{L^4} \|\nabla\partial^\alpha\theta\|_{L^2}. 
\end{split}
\]
Together with $\eqref{EstClassical}_1$, we deduce 
\begin{equation}\label{EqD^alphaULeft_4}
\begin{split}
\sum_{ |\beta| =1} \int\partial^{\alpha-\beta}u \cdot\nabla \partial^\beta(\rho_0+\theta) \Delta\partial^{\alpha}\theta &\gtrsim  -\|u\|_{H^2}\|\nabla\partial^\alpha \theta\|_{L^2} -   \|u\|_{H^3}\|\theta\|_{H^3}\|\nabla\partial^\alpha \theta\|_{L^2}\\
&\gtrsim -\|u\|_{H^2}\|\theta\|_{H^3} - \|u\|_{H^3}\|\theta\|_{H^3}^2.
\end{split}
\end{equation}
Combining \eqref{EqD^alphaULeft_1}, \eqref{EqD^alphaULeft_2} and \eqref{EqD^alphaULeft_3}, \eqref{EqD^alphaULeft_4} gives
\begin{equation}\label{EstD^alphaULeft}
\begin{split}
&\int\partial^{\alpha}((\rho_0+\theta) \partial_t u) \cdot \partial^{\alpha}u +\sigma \int\partial^{\alpha}(\nabla(\rho_0+\theta) \Delta\theta) \cdot\partial^{\alpha}u  - \mu\int \Delta\partial^\alpha u \cdot \partial^{\alpha} u  \\
&\gtrsim \frac{d}{dt} \int( (\rho_0+\theta) |\partial^{\alpha}u|^2 + |\partial^{\alpha}\nabla\theta|^2) + \int |\nabla\partial^\alpha u|^2  -  \|u\|_{H^2}\|\theta\|_{H^3}-\|\partial_t u\|_{H^1}\|u\|_{H^2} -\cE^3
\end{split}
\end{equation}

We now estimate the r.h.s of \eqref{EqD^alphaU}. 
\[
\begin{split}
&\int\partial^{\alpha}((\rho_0+\theta)u\cdot\nabla u) \cdot \partial^{\alpha} u \\
&= \int (\partial^{\alpha}((\rho_0+\theta)u\cdot\nabla u) -(\rho_0+\theta) u\cdot\nabla\partial^\alpha u) \cdot \partial^{\alpha} u + \int (\rho_0+\theta) (u\cdot\nabla\partial^\alpha u) \cdot \partial^\alpha u \\
&= \sum_{0\neq \beta\leq \alpha} \int ( \partial^\beta( (\rho_0+\theta)u) \cdot \partial^{\alpha-\beta}\nabla u) \cdot \partial^\alpha u + \frac12\int (\rho_0+\theta) u \cdot \nabla|\partial^\alpha u|^2.
\end{split}
\]
We use H\"older's inequality and Sobolev embedding to have  
\[
\begin{split}
&\sum_{0\neq \beta\leq \alpha} \int ( \partial^\beta( (\rho_0+\theta)u) \cdot \partial^{\alpha-\beta}\nabla u) \cdot \partial^\alpha u \\
&\lesssim  (\|\nabla( (\rho_0+\theta)u)\|_{H^2}\| \nabla^2 u\|_{L^2}+ \|\nabla^2((\rho_0+\theta)u)\|_{L^4} \|\nabla u\|_{L^4})\|\nabla^2 u\|_{L^2} \\
&\lesssim \|(\rho_0+\theta)u\|_{H^3} \|u\|_{H^2}^2.
\end{split}
\]
We get further
\[\begin{split}
\sum_{0\neq \beta\leq \alpha} \int ( \partial^\beta( (\rho_0+\theta)u) \cdot \partial^{\alpha-\beta}\nabla u) \cdot \partial^\alpha u &\lesssim \| \rho_0+\theta\|_{H^3} \|u\|_{H^3} \|u\|_{H^3}^2   \\
&\lesssim (1+\|\theta\|_{H^3})\|u\|_{H^3}^3.
\end{split}
\]
Thanks to the integration by parts and Sobolev embedding, we have
\begin{equation}\label{EqD^alphaURight_1}
\begin{split}
\int (\rho_0+\theta) u \cdot \nabla|\partial^\alpha u|^2= - \int (u\cdot \nabla(\rho_0+\theta))|\partial^\alpha u|^2 &\lesssim \|u\|_{H^2} \|\rho_0+\theta\|_{H^3}\|u\|_{H^2}^2 \\
&\lesssim (1+\|\theta\|_{H^3})\|u\|_{H^3}^3.
\end{split}
\end{equation}
We also have
\begin{equation}\label{EqD^alphaURight_2}
\begin{split}
\int\partial^{\alpha}(\rho_0'' \nabla\theta)\cdot \partial^{\alpha} u  \lesssim \|\nabla\theta\|_{H^2}\|u\|_{H^2} \lesssim \|\theta\|_{H^3}\|u\|_{H^2}.
\end{split}
\end{equation}
In view of \eqref{EqD^alphaURight_1} and \eqref{EqD^alphaURight_2}, we have
\begin{equation}\label{EstD^alphaURight}
\begin{split}
&\int\partial^{\alpha}((\rho_0+\theta)u\cdot\nabla u) \cdot \partial^{\alpha} u  +\sigma \int\partial^{\alpha}(\rho_0'' \nabla\theta)\cdot \partial^{\alpha} u  + g\int\partial^{\alpha}\theta \partial^{\alpha}u_3 \\
&\qquad\lesssim \| u\|_{H^2} \|\theta\|_{H^3} + (1+\|\theta\|_{H^3})\|u\|_{H^3}^3.
\end{split}
\end{equation}
We combine \eqref{EstD^alphaULeft}, \eqref{EstD^alphaURight} and \eqref{BoundDtUinH^1byE} to obtain that 
\begin{equation}\label{EstNormU_H3_Theta_H4_1}
\begin{split}
 \frac{d}{dt} \int( (\rho_0+\theta) |\partial^{\alpha}u|^2 + |\partial^{\alpha}\nabla\theta|^2) + \int |\nabla\partial^\alpha u|^2 &\lesssim   \|u\|_{H^2}^2 +(\|\theta\|_{H^3}+ \|\partial_t u\|_{H^1}) +\cE^3 \\
 &\lesssim \|u\|_{H^2}\cE+\cE^3.
\end{split}
\end{equation}
It follows from $\eqref{EstClassical}_3$ that  $\|u\|_{H^2} \lesssim \varepsilon \|u\|_{H^3}+\varepsilon^{-2} \|u\|_{L^2}$.
It yields
\[
 \frac{d}{dt} \int( (\rho_0+\theta) |\partial^{\alpha}u|^2 + |\partial^{\alpha}\nabla\theta|^2) + \int |\nabla\partial^\alpha u|^2 \lesssim  \varepsilon \cE^2 +\varepsilon^{-2}\|u\|_{L^2}^2 +\cE^3.
\]
Integrating in time from $0$ to $t$, we deduce 
\[
\begin{split}
&\int (\rho_0+\theta(t)) |\partial^{\alpha}u(t)|^2 + |\partial^{\alpha}\nabla\theta(t)|^2) +\int_0^t \int |\nabla\partial^\alpha u(s)|^2 ds \\
&\qquad\lesssim \cE^2(0) +  \varepsilon \int_0^t \cE^2(s)ds + \varepsilon^{-2} \int_0^t \|u(s)\|_{L^2}^2 ds+ \int_0^t \cE^3(s)ds.
\end{split}
\]
Summing over $0\neq \alpha \in \N^3$ and chaining with \eqref{EstNormU_L2_Theta_H1}, the inequality \eqref{EstNormU_H3_Theta_H4} follows. Proposition \ref{PropNormU_H3_Theta_H4} is proven. 
\end{proof}

We apply the classical regularity theory on the Stokes equations to obtain further some elliptic estimates.

\begin{proposition}\label{PropEllipticEst}
There holds
\begin{equation}\label{EstElliptic}
\|\nabla^2 u\|_{H^1} +\|\nabla q\|_{H^1} \lesssim  \|\partial_t u\|_{H^1} +\|\theta\|_{H^3}+\cE^2.
\end{equation}
\end{proposition}
\begin{proof}
We rewrite $\eqref{EqNSK_Pertur}_2$ as 
\begin{equation}
\begin{split}
-\mu\Delta u +\nabla q= -(\rho_0+\theta) \partial_t u -(\rho_0+\theta) u\cdot \nabla u -\sigma ( \nabla(\rho_0+\theta) \Delta\theta +\rho_0'' \nabla\theta)- g\theta e_3.
\end{split}
\end{equation}
Applying the classical regularity theory on the Stokes equations to the resulting equation, we have
\begin{equation}
\begin{split}
\|\nabla^2 u\|_{H^1}+ \|\nabla q\|_{L^2} &\lesssim \|(\rho_0+\theta)\partial_t u\|_{H^1} + \|(\rho_0+\theta) u\cdot\nabla u\|_{H^1} \\
&\qquad+ \|\nabla(\rho_0+\theta) \Delta\theta\|_{H^1} + \|\nabla\theta\|_{H^1}+\|\theta\|_{H^1} \\
&\lesssim (1+\|\theta\|_{H^2})\|\partial_t u\|_{H^1}+\|\theta\|_{H^3} +\cE^2 \\
&\lesssim \|\partial_t u\|_{H^1} +\|\theta\|_{H^3}+\cE^2.
\end{split} 
\end{equation}
Hence, \eqref{EstElliptic} is established.
\end{proof}

Thanks to Propositions \ref{PropEllipticEst}, we are able to prove Proposition \ref{PropEstimates}.
\begin{proof}[Proof of Proposition \ref{PropEstimates}]  
Combining the two inequalities \eqref{EstIntDtU_L2} and \eqref{EstIntDt^2U_L2}  from Proposition \ref{PropNormU_DtU_H1}, we have 
\begin{equation}\label{0_EstApriori}
\begin{split}
&\|\partial_t u(t)\|_{H^1}^2 + \|\partial_t\theta(t)\|_{L^2}^2 + \int_0^t (\|\partial_t u(s)\|_{H^1}^2 + \|\partial_t^2 u(s)\|_{L^2}^2) ds \\
&\qquad\lesssim  \cE^2(0)+ \int_0^t (\varepsilon \|u(s)\|_{H^3}^2 + \varepsilon^{-2}\|u(s)\|_{L^2}^2 +\cE^4(s))ds
\end{split}
\end{equation}
In view of \eqref{0_EstApriori} and the estimate \eqref{EstNormU_H3_Theta_H4} from Proposition \ref{PropNormU_H3_Theta_H4}, we get
\[
\begin{split}
&\|u(t)\|_{H^2}^2 + \|\theta(t)\|_{H^3}^2 +\|\partial_t u(t)\|_{H^1}^2 + \|\partial_t\theta(t)\|_{L^2}^2 \\
&\qquad\qquad+ \int_0^t (\|\nabla u(s)\|_{H^2}^2+ \|\partial_t u(s)\|_{H^1}^2 + \|\partial_t^2 u(s)\|_{L^2}^2) ds \\
&\lesssim \cE^2(0)+\varepsilon  \int_0^t \cE^2(s) ds+ \varepsilon^{-2} \int_0^t\|(\theta,u)(s)\|_{L^2}^2 ds+  \int_0^t \cE^3(s)ds.
\end{split}
\]
The resulting inequality and the estimate  \eqref{EstElliptic} from Proposition \ref{PropEllipticEst} yield
\begin{equation}\label{2_EstApriori}
\begin{split}
&\|u(t)\|_{H^2}^2 + \|\theta(t)\|_{H^3}^2 +\|\partial_t u(t)\|_{H^1}^2 + \|\partial_t\theta(t)\|_{L^2}^2 +\varepsilon^{1/2} (\|\nabla^2u(t)\|_{H^1}+\|\nabla q(t)\|_{L^2}^2) \\
&\qquad\qquad+ \int_0^t (\|\nabla u(s)\|_{H^2}^2+ \|\partial_t u(s)\|_{H^1}^2 + \|\partial_t^2 u(s)\|_{L^2}^2) ds \\
&\lesssim \varepsilon^{1/2} (\|\partial_tu(t)\|_{H^1}^2 +\|\theta(t)\|_{H^3}^2+\cE^4(t)) +  \cE^2(0)\\
&\qquad\qquad+\varepsilon  \int_0^t\cE^2(s) ds+ \varepsilon^{-2} \int_0^t\|(\theta,u)(s)\|_{L^2}^2 ds+  \int_0^t \cE^3(s)ds.
\end{split}
\end{equation}
We decrease $\varepsilon$ if necessary to obtain from \eqref{2_EstApriori} that
\[
\begin{split}
&\|u(t)\|_{H^3}^2 + \|\theta(t)\|_{H^3}^2 +\|\partial_t u(t)\|_{H^1}^2 + \|\partial_t\theta(t)\|_{L^2}^2 +\|\nabla q(t)\|_{L^2}^2 \\
&\qquad\qquad+ \int_0^t (\|\nabla u(s)\|_{H^2}^2+ \|\partial_t u(s)\|_{H^1}^2 + \|\partial_t^2 u(s)\|_{L^2}^2) ds \\
&\lesssim \cE^4(t) + \varepsilon^{-1/2} \cE^2(0)+\varepsilon^{1/2}  \int_0^t\cE^2(s) ds+ \varepsilon^{-5/2} \int_0^t\|(\theta,u)(s)\|_{L^2}^2 ds+ \varepsilon^{-1/2} \int_0^t \cE^3(s)ds.
\end{split}
\]
This implies 
\begin{equation}\label{3_EstApriori}
\begin{split}
&\cE^2(t) +\|\partial_t u(t)\|_{H^1}^2 + \|\partial_t\theta(t)\|_{L^2}^2 +\|\nabla q(t)\|_{L^2}^2 + \int_0^t (\|\nabla u(s)\|_{H^2}^2+ \|\partial_t u(s)\|_{H^1}^2 + \|\partial_t^2 u(s)\|_{L^2}^2) ds \\
&\lesssim \delta_0^2 \cE^2(t) + \varepsilon^{-1/2} \cE^2(0)+\varepsilon^{1/2}  \int_0^t\cE^2(s) ds+ \varepsilon^{-5/2} \int_0^t\|(\theta,u)(s)\|_{L^2}^2 ds+ \varepsilon^{-1/2} \int_0^t \cE^3(s)ds.
\end{split}
\end{equation}
If $\delta_0$ is taken small enough, the inequality \eqref{3_EstApriori} yields
\begin{equation}\label{4_EstApriori}
\begin{split}
&\cE^2(t) +\|\partial_t u(t)\|_{H^1}^2 + \|\partial_t\theta(t)\|_{L^2}^2 +\|\nabla q(t)\|_{L^2}^2 + \int_0^t (\|\nabla u(s)\|_{H^2}^2+ \|\partial_t u(s)\|_{H^1}^2 + \|\partial_t^2 u(s)\|_{L^2}^2) ds \\
&\lesssim  \varepsilon^{-1/2} \cE^2(0)+\varepsilon^{1/2}  \int_0^t\cE^2(s) ds+ \varepsilon^{-5/2} \int_0^t\|(\theta,u)(s)\|_{L^2}^2 ds+ \varepsilon^{-1/2} \int_0^t \cE^3(s)ds.
\end{split}
\end{equation}
Let us change $\varepsilon^{1/2}$ by $\varepsilon$ in \eqref{4_EstApriori}, the inequality \eqref{EstApriori} thus follows. The proof of Proposition \ref{PropEstimates} is finished. 
\end{proof}

\subsection{Proof of Theorem \ref{ThmUnstable}}\label{SectNonlinearInstability}

As presented in Section \ref{SectMainResults}, let us consider the nonlinear equations \eqref{EqDiff} with the initial data $(\theta^d, u^d)(0)=0$. The aim of this section is to derive a bound in time of $(\theta^d, u^d)$.

Let 
\[
F_\sN(t)=\sum_{j=j_m}^\sN |\Csf_j| e^{\lambda_j t}
\]
and $0<\varepsilon_0\ll 1$ be fixed later. There exists a unique $T^\delta$ such that $\delta F_\sN(T^\delta)=\varepsilon_0$. Let 
\begin{equation}\label{DefC_12}
C_1 := \sqrt{\|\theta^\sN(0)\|_{H^3}^2+\|u^\sN(0)\|_{H^3}^2}, \quad C_2 := \sqrt{\|\theta^\sN(0)\|_{L^2}^2+\|u^\sN(0)\|_{L^2}^2},
\end{equation}
we define 
\begin{equation}\label{DefTstar}
\begin{split}
T^\star &:= \sup\{ t\in (0, T^{\max}), \cE(\theta^\delta(t), u^\delta(t)) \leq C_1\delta_0\}, \\
T^{\star\star} &:=  \sup\{ t\in (0, T^{\max}),\|( \theta^\delta, u^\delta)(t)\|_{L^2} \leq 2C_2\delta F_\sN(t) \}.
\end{split}
\end{equation}
Note that $\cE(\theta^\delta(0), u^\delta(0)) =C_1\delta <C_1\delta_0$, hence $T^\star>0$. Similarly, we have $T^{\star\star}>0$. Then for all $t\leq \min\{T^\delta, T^\star, T^{\star\star}\}$, we derive the following bound in time of 
$\cE(\theta^\delta(t), u^\delta(t))$.
\begin{proposition}\label{PropNormUdelta}
For all $t\leq \min\{T^\delta, T^\star, T^{\star\star}\}$, there holds
\begin{equation}\label{NormUdelta}
\begin{split}
\|\theta^\delta(t)\|_{H^3}+ \|u^\delta(t)\|_{H^3}+\|\partial_t u^\delta(t)\|_{H^1} + \|\partial_t\theta^\delta(t)\|_{L^2}  \leq C_3 \delta F_\sN(t).
\end{split}
\end{equation}
\end{proposition}
\begin{proof}
For short, we write $\cE_\delta(t)$ instead of $\cE(\theta^\delta(t), u^\delta(t))$. It follows from the \textit{a priori} energy estimate \eqref{EstApriori} that 
\begin{equation}\label{0_NormUdelta}
\begin{split}
&\cE_\delta^2(t) +\|\partial_t u^\delta(t)\|_{H^1}^2 + \|\partial_t\theta^\delta(t)\|_{L^2}^2 +\|\nabla q^\delta(t)\|_{L^2}^2 \\
&\qquad + \int_0^t (\|\nabla u^\delta(s)\|_{H^2}^2+ \|\partial_t u^\delta(s)\|_{H^1}^2 + \|\partial_t^2 u^\delta(s)\|_{L^2}^2) ds \\
&\leq C_3\Big( \varepsilon^{-1} \cE_\delta^2(0)+\varepsilon  \int_0^t\cE_\delta^2(s) ds+ \varepsilon^{-5} \int_0^t\|(\theta^\delta,u^\delta)(s)\|_{L^2}^2 ds+ \varepsilon^{-1} \int_0^t \cE_\delta^3(s)ds\Big), 
\end{split}
\end{equation}
Let us decrease $C_3\varepsilon \leq \frac{\lambda_\sN}2$, so that 
\[
\begin{split}
\cE_\delta^2(t) &\leq C_{\lambda_\sN}\delta^2+ \frac{\lambda_\sN}2 \int_0^t \cE_\delta^2(s)ds + C_{\lambda_\sN} \int_0^t \delta^2 F_\sN^2(s)ds +  C_{\lambda_\sN} \int_0^t \cE_\delta^3(s)ds \\
&\leq  \Big( \frac{\lambda_\sN}2+ C_{\lambda_\sN}\delta\Big) \int_0^t \cE_\delta^2(s)ds + C_4 \delta^2F_\sN^2(t). 
\end{split}
\]
Refining $\delta_0$ such that $C_{\lambda_\sN}\delta_0 \leq \frac{\lambda_\sN}2$, we observe 
\[
\cE_\delta^2(t)  \leq \lambda_\sN \int_0^t \cE_\delta^2(s)ds+ C_4 \delta^2F_\sN^2(t).
\]
Applying Gronwall's inequality, we have
\begin{equation}\label{1_NormUdelta}
\cE_\delta^2(t) \lesssim \delta^2 F_\sN^2(t) +\delta^2  \int_0^t e^{\lambda_\sN(t-s)}  F_\sN^2(s)ds.
\end{equation}
Note that $\lambda_\sN<\lambda_j$ for any $1\leq j<\sN$. Hence 
\begin{equation}\label{2_NormUdelta}
 \int_0^t e^{\lambda_\sN(t-s)}  F_\sN^2(s)ds \lesssim\sum_{j=j_m}^\sN \int_0^t |\Csf_j|^2 e^{\lambda_\sN t}  e^{ (2\lambda_j-\lambda_\sN) s} ds \lesssim \sum_{j=j_m}^\sN |\Csf_j|^2 \frac{e^{2\lambda_j t}}{2\lambda_j-\lambda_\sN}. 
\end{equation}
Substituting \eqref{2_NormUdelta} into \eqref{1_NormUdelta}, we  deduce that $\cE_\delta(t) \lesssim \delta F_\sN(t)$. Putting it back to \eqref{0_NormUdelta}, we conclude that \eqref{NormUdelta} holds. 
\end{proof}

Thanks to  Proposition \ref{PropNormUdelta}, we derive the following bound in time of $\|(\theta^d, u^d)(t)\|_{L^2}$.
\begin{proposition}\label{PropBoundDiff}
There holds 
\begin{equation}\label{BoundDiff}
\|\theta^d(t)\|_{L^2}^2+\| u^d(t) \|_{L^2}^2 \leq C_4\delta^3  \Big( \sum_{j=j_m}^\sK |\Csf_j| e^{\lambda_j t} +  \sum_{j=\sK+1}^\sN  |\Csf_j| e^{\frac23 \lambda_1 t}\Big)^3.
\end{equation}
\end{proposition}
To prove Proposition \ref{PropBoundDiff}, we need the  following lemma.
\begin{lemma}\label{LemDtUd(0)}
There holds 
\begin{equation}\label{BoundDtUd0}
\|\partial_t u^d(0)\|_{L^2}^2 \lesssim \delta^3.
\end{equation}
\end{lemma}
\begin{proof}
Due to the incompressibility condition, it follows from $\eqref{EqDiff}_2$ that 
\[
\begin{split}
\int \rho_0 |\partial_t u^d|^2 &= \int (\mu \Delta u^d -\sigma(\rho_0' \Delta\theta^d e_3+\rho_0''\nabla\theta^d)-g\theta^d)\cdot\partial_t u^d \\
&\qquad - \int (\theta^\delta\partial_t u^\delta-(\rho_0+\theta^\delta) u^\delta\cdot\nabla u^\delta -\sigma \nabla\theta^\delta\Delta\theta^\delta) \cdot\partial_t u^d.
\end{split}
\]
For any $\nu>0$, thanks to Young's inequality, we obtain 
\begin{equation}\label{1_BoundDtUd0}
\begin{split}
 \int (\mu \Delta u^d -\sigma(\rho_0' \Delta\theta^d e_3+\rho_0''\nabla\theta^d)-g\theta^d)\cdot\partial_t u^d \leq \nu \|\partial_t u^d\|_{L^2}^2 + \nu^{-1} (\|\Delta u^d\|_{L^2}+ \|\theta^d\|_{H^2})^2.
\end{split}
\end{equation}
Using the interpolation inequality also, we have
\[
\begin{split}
\int& (\theta^\delta\partial_t u^\delta -(\rho_0+\theta^\delta) u^\delta\cdot\nabla u^\delta -\sigma \nabla\theta^\delta\Delta\theta^\delta) \cdot\partial_t u^d \\
& \leq   (\|\theta^\delta\partial_t u^\delta\|_{L^2} +\|(\rho_0+\theta^\delta) u^\delta\cdot\nabla u^\delta\|_{L^2}+ \|\nabla\theta^\delta\Delta\theta^\delta\|_{L^2})\|\partial_t u^d\|_{L^2} \\
&\lesssim   ( \|\theta^\delta\|_{H^2}\|\partial_t u^\delta\|_{L^2} + (1+\|\theta^\delta\|_{H^2})\|u^\delta\|_{H^2}^2+ \|\nabla\theta^\delta\|_{L^4} \|\Delta\theta^\delta\|_{L^4})(\|\partial_t u^\delta\|_{L^2} +\delta \|\partial_t u^\sN\|_{L^2}).
\end{split}
\]
Together with \eqref{NormUdelta}, this implies 
\begin{equation}\label{2_BoundDtUd0}
\begin{split}
\int (\theta^\delta\partial_t  u^\delta -(\rho_0+\theta^\delta) u^\delta\cdot\nabla u^\delta -\sigma \nabla\theta^\delta\Delta\theta^\delta) \cdot\partial_t u^d \lesssim  \delta^3 F_\sN^3.
\end{split}
\end{equation}
Owing to \eqref{1_BoundDtUd0} and \eqref{2_BoundDtUd0} with $\nu$ sufficiently small, we have
\[
 \|\partial_t u^d(t)\|_{L^2}^2 \lesssim \|\Delta u^d(t)\|_{L^2}^2+ \|\theta^d(t)\|_{H^2}^2 + \delta^3 F_\sN^3(t). 
\]
Letting $t\to 0$, we deduce \eqref{BoundDtUd0}. 
\end{proof}

Now, we are in position to prove  of Proposition \ref{PropBoundDiff}.
\begin{proof}[Proof of Proposition \ref{PropBoundDiff}]
Let us write $\eqref{EqDiff}_2$ as 
\begin{equation}
(\rho_0+\theta^\delta)\partial_t u^d +\nabla q^d -\mu\Delta u^d + \sigma (\rho_0'\Delta\theta^d e_3+ \rho_0''\nabla\theta^d)  = f^\delta -g\theta^d e_3,
\end{equation}
where $f^\delta = \delta \theta^\delta \partial_tu^\sN - (\rho_0+\theta^\delta)u^\delta \cdot\nabla u^\delta -\sigma \nabla\theta^\delta \Delta\theta^\delta   $. Differentiate the resulting equation with respect to $t$ and then multiply by $\partial_t u^d$, we obtain after integration that 
\[
\begin{split}
&\int \partial_t\theta^\delta |\partial_t u^d|^2 + \int (\rho_0+\theta^\delta)\partial_t^2 u^d\cdot\partial_t u^d  +\sigma \int (\rho_0' \partial_t u_3^d \Delta\partial_t \theta^d +\rho_0''\nabla\partial_t\theta^d \cdot\partial_t u^d) \\
&= \int (\mu\Delta\partial_t u^d-\nabla\partial_t p^d)\cdot\partial_t u^d +\int \partial_t f^\delta\cdot\partial_t u^d - g\int\partial_t\theta^d \partial_t u_3^d.
\end{split}
\]
Using $\eqref{EqDiff}_1$ and \eqref{LemLambda_5}-\eqref{LemLambda_6}, it can be seen that 
\[
\begin{split}
&\int (\rho_0' \partial_t u_3^d \Delta\partial_t \theta^d +\rho_0''\nabla\partial_t\theta^d \cdot\partial_t u^d) \\
&=- \int \rho_0' u_3^d \Delta(\rho_0'u_3^d+ u^\delta\cdot\nabla \theta^\delta) + \int (\rho_0'u_3^d+ u^\delta\cdot\nabla\theta^\delta) \text{div}(\rho_0'' \partial_t u^d) \\
&= \frac12 \frac{d}{dt} \int( |\nabla(\rho_0' u_3^d)|^2 +\rho_0'\rho_0'''|u_3^d|^2)- \int(\rho_0'u_3^d +\rho_0'''\partial_t u_3^d) u^\delta\cdot\nabla \theta^\delta \\
&= \frac12 \frac{d}{dt} \int |\rho_0'| |\nabla u_3^d|^2 -  \int(\rho_0'u_3^d +\rho_0'''\partial_t u_3^d) u^\delta\cdot\nabla \theta^\delta.
\end{split}
\]
This implies 
\[
\begin{split}
&\frac12 \frac{d}{dt}\Big( \int  (\rho_0+\theta^\delta)|\partial_tu^d|^2 -g\int \rho_0'|u_3^d|^2 +\sigma\int ( |\nabla(\rho_0'u_3^d)|^2 + \rho_0'\rho_0''' |u_3^d|^2)  \Big) + \mu \int|\nabla\partial_t u^d|^2   \\
&= -\frac12 \int \partial_t\theta^\delta |\partial_tu^d|^2 + g\int \partial_t u_3^d u^\delta\cdot\nabla\theta^\delta + \int \partial_t f^\delta \cdot\partial_t u^d  + \sigma  \int(\rho_0'u_3^d +\rho_0'''\partial_t u_3^d) u^\delta\cdot\nabla \theta^\delta.
\end{split}
\]
Note that $u^d(0)=0$, integrating in time from 0 to $t$ yields 
\begin{equation}\label{1_PropBoundDiff}
\begin{split}
&\int (\rho_0+\theta^\delta(t)) |\partial_t u^d(t)|^2 + \mu\int_0^t \int |\nabla\partial_t u^d(s)|^2 ds\\
&= \Big(\int (\rho_0+\theta^\delta(t)) |\partial_t u^d(t)|^2\Big) \Big|_{t=0}+  g\int\rho_0'|u_3^d(t)|^2 - \sigma \int (\rho_0')^2 |\nabla u_3^d(t))|^2  \\
&\qquad - \int_0^t  \Big( \int \partial_t\theta^\delta |\partial_tu^d|^2 -2\int (\partial_t f^\delta \cdot\partial_t u^d  + (g\partial_t u_3^d+   \sigma  (\rho_0'u_3^d +\rho_0'''\partial_t u_3^d) u^\delta\cdot\nabla \theta^\delta)\Big)(s) ds.
\end{split}
\end{equation}
We now estimate the r.h.s of \eqref{1_PropBoundDiff}. Due to Sobolev embedding and three inequalities \eqref{BoundDtUinH^1byE}, \eqref{BoundDtThetaH2} and \eqref{NormUdelta}, we estimate that
\begin{equation}\label{2_PropBoundDiff}
\begin{split}
\int \partial_t\theta^\delta |\partial_tu^d|^2 \lesssim \|\partial_t\theta^\delta\|_{H^2} \|\partial_t u^d\|_{L^2}^2 &\lesssim \|u^\delta\|_{H^2}(\|\partial_t u^\delta\|_{L^2}+ \delta \|\partial_t u^\sN\|_{L^2})^2\\
&\lesssim \delta^3 F_\sN^3,
\end{split}
\end{equation}
and
\begin{equation}\label{3_PropBoundDiff}
\begin{split}
\int (g\partial_t u_3^d+   \sigma  (\rho_0'u_3^d +\rho_0'''\partial_t u_3^d))u^\delta \cdot \nabla\theta^\delta &\lesssim \|(u_3^d,\partial_t u_3^d)\|_{L^2}\|u^\delta\|_{H^2}\|\nabla\theta^\delta\|_{L^2}\\
&\lesssim(\|(u^\delta, \partial_t u^\delta)\|_{L^2}+\delta \|(u^\sN, \partial_t u^\sN)\|_{L^2}) \|u^\delta\|_{H^2}\|\theta^\delta\|_{H^1} \\
&\lesssim \delta^3 F_\sN^3.
\end{split}
\end{equation}
Next, let us estimate $\|\partial_t f^\delta\|_{L^2}$ as follows. We use Sobolev embedding, \eqref{BoundDtUinH^1byE}, \eqref{BoundDtThetaH2} and \eqref{NormUdelta} again to have that
\begin{equation}\label{4_PropBoundDiff}
 \|\partial_t(\theta^\delta \partial_t u^\sN)\|_{L^2} \lesssim \|\partial_t\theta^\delta \|_{L^2} \|\partial_t u^\sN\|_{H^2} + \|\theta^\delta\|_{L^2} \|\partial_t^2 u^\sN\|_{H^2}  \lesssim \delta F_\sN^2,
\end{equation}
that
\begin{equation}\label{5_PropBoundDiff}
\begin{split}
\|\partial_t ((\rho_0+ \theta^\delta) u^\delta \cdot\nabla u^\delta )\|_{L^2} \lesssim \|\partial_t \theta^\delta\|_{L^2} \|u^\delta\|_{H^3}^2 + (1+\|\theta^\delta\|_{H^2}) \|\partial_t u^\delta\|_{H^1}\|u^\delta\|_{H^3} \lesssim \delta^2 F_\sN^2,
\end{split}
\end{equation}
and that
\begin{equation}\label{6_PropBoundDiff}
\begin{split}
\|\partial_t(\nabla\theta^\delta\Delta \theta^\delta)\|_{L^2} \lesssim \| \partial_t\nabla\theta^\delta \|_{H^2}  \|\Delta\theta^\delta\|_{L^2} + \| \nabla\theta^\delta\|_{H^2} \|\partial_t \Delta\theta^\delta\|_{L^2} &\lesssim \|\partial_t\theta^\delta\|_{H^3}\|\theta^\delta\|_{H^3} \\
&\lesssim \|u^\delta\|_{H^3}\|\theta^\delta\|_{H^3}\\
&\lesssim \delta^2F_\sN^2.
\end{split}
\end{equation}
It follows from \eqref{2_PropBoundDiff}, \eqref{3_PropBoundDiff}, \eqref{4_PropBoundDiff}, \eqref{5_PropBoundDiff} and \eqref{6_PropBoundDiff} that 
\[
\begin{split}
\int \rho_0 |\partial_t u^d(t)|^2 + \mu\int_0^t \int |\nabla\partial_t u^d(s)|^2 &\leq \Big(\int (\rho_0+\theta^\delta(t)) |\partial_t u^d(t)|^2\Big) \Big|_{t=0}+  g\int\rho_0'|u_3^d(t)|^2 \\
&\qquad - \sigma \int ( |\nabla(\rho_0'u_3^d(t))|^2 + \rho_0'\rho_0''' |u_3^d(t)|^2) +C\delta^3 F_\sN^3(t), 
\end{split}
\]
where $C$ is a generic constant. 
Thanks to Lemmas \ref{LemDtUd(0)}, \ref{LemMaxLambda}, we obtain further
\[
\begin{split}
\int \rho_0 |\partial_t u^d(t)|^2 + \mu\int_0^t \int |\nabla\partial_t u^d(s)|^2 ds \leq \Lambda^2 \int \rho_0|u^d(t)|^2 + \Lambda\mu \int |\nabla u^d(t)|^2 +C \delta^3 F_\sN^3(t).
\end{split}
\]
Estimate as same as \eqref{LemLambda_2}-\eqref{LemLambda_3}, we get 
\[
\begin{split}
\frac{d}{dt} \|\sqrt{\rho_0} u^d(t)\|_{L^2}^2 +\mu \|\nabla u^d(t)\|_{L^2}^2 &\leq 2\Lambda\Big( \|\sqrt{\rho_0}u^d(t)\|_{L^2}^2 + \mu \int_0^t \|\nabla u^d(s)\|_{L^2}^2 ds\Big) + C\delta^3 F_\sN^3(t). 
\end{split}
\]
Applying Gronwall's inequality, we obtain 
\[
\begin{split}
 \|\sqrt{\rho_0} u^d(t)\|_{L^2}^2 +\mu \int_0^t \|\nabla u^d(s)\|_{L^2}^2 ds &\lesssim \delta^3 e^{2\Lambda t} \int_0^t e^{-2\Lambda s} F_\sN^3(s)ds \\
&\lesssim  \delta^3 e^{2\Lambda t} \sum_{j=j_m}^\sN \int_0^t |\Csf_j|^3 e^{(3\lambda_j -2\Lambda)s}ds. 
\end{split}
\]
For each $1\leq j\leq \sK$, we have $\lambda_j > \frac23 \Lambda$, yielding 
\[
\int_0^t e^{(3\lambda_j -2\Lambda)s}ds = \frac{e^{(3\lambda_j-2\Lambda) t}-1}{3\lambda_j-2\Lambda} \lesssim e^{(3\lambda_j-2\Lambda) t},
\]
and for each $\sK+1\leq j\leq \sN$, we have $\lambda_j < \frac23 \Lambda$, yielding
\[
\int_0^t e^{(3\lambda_j -2\Lambda)s}ds =  \frac{e^{(3\lambda_j-2\Lambda) t}-1}{3\lambda_j-2\Lambda} \lesssim 1.
\]
Consequently, 
\begin{equation}\label{BoundUdL2}
\begin{split}
 \|\sqrt{\rho_0} u^d(t)\|_{L^2}^2 +\mu \int_0^t \|\nabla u^d(s)\|_{L^2}^2 ds &\lesssim \delta^3  \Big( \sum_{j=j_m}^\sK |\Csf_j|^3 e^{3\lambda_j t} + \sum_{j=\sK+1}^\sN  |\Csf_j|^3 e^{2\Lambda t}\Big).
\end{split}
\end{equation}

To show the bound of $\|\theta^d(t)\|_{L^2}$, we use Sobolev embedding to deduce $\eqref{EqDiff}_1$ that 
\[
\begin{split}
\frac{d}{dt} \|\theta^d\|_{L^2} \leq \|\theta^d\|_{L^2} &\leq \max\rho_0' \|u_3^d\|_{L^2} + \|u^\delta\cdot\nabla \theta^\delta\|_{L^2} \\
&\lesssim \|u_3^d\|_{L^2} + \|u^\delta\|_{H^2} \|\theta^\delta\|_{H^1}. 
\end{split}
\]
Using \eqref{NormUdelta}, we obtain further
\[
\frac{d}{dt} \|\theta^d\|_{L^2} \lesssim \|u_3^d\|_{L^2} + \delta^2 F_\sN^2. 
\]
Note that $\theta^d(0)=0$. Integrating in time from 0 to $t$ and using \eqref{BoundUdL2}, it thus follows that $\|\theta^d(t)\|_{L^2}$ is also bounded above as same as $\|u^d(t)\|_{L^2}$. Proof of Proposition \ref{PropBoundDiff} is complete. 
\end{proof}

We are in position to prove Theorem \ref{ThmUnstable}.
\begin{proof}[Proof of Theorem \ref{ThmUnstable}]
Note that 
\begin{equation}
\begin{split}
\| u^\sN(t)\|_{L^2}^2 &=   \sum_{i=j_m}^\sN \Csf_i^2 e^{2\lambda_i t} \| v_i\|_{L^2}^2 +  2 \sum_{j_m\leq i<j\leq \sN} \Csf_i \Csf_j e^{(\lambda_i+\lambda_j)t} \int  v_i  \cdot  v_j .
\end{split}
\end{equation}
It can be seen that 
\[
\begin{split}
\| u^\sN(t)\|_{L^2}^2 \geq  &\sum_{j=j_m}^\sN \Csf_j^2 e^{2\lambda_j t}\| v_j\|_{L^2}^2 + 2\sum_{j_m+1\leq i<j\leq \sN} \Csf_i\Csf_j e^{(\lambda_i+\lambda_j)t}   \int  v_i\cdot  v_j\\
&\qquad-|\Csf_{j_m}|\| v_{j_m}\|_{L^2} \Big(\sum_{j=j_m+1}^\sN |\Csf_j|  \| v_j\|_{L^2}\Big) e^{(\lambda_{j_m}+\lambda_{j_m+1})t}.
\end{split}
\]
By Cauchy-Schwarz's inequality, we obtain 
\[
\begin{split}
2\sum_{j_m+1 \leq i<j\leq \sN} \Csf_i\Csf_j e^{(\lambda_i+\lambda_j)t}   \int v_i\cdot v_j &\geq- 2\sum_{j_m+1 \leq i<j\leq \sN} |\Csf_i||\Csf_j| e^{(\lambda_i+\lambda_j)t} \|v_i\|_{L^2}\| v_j\|_{L^2}\\
&\geq - e^{(\lambda_{j_m+1}+\lambda_{j_m+2})t} \Big(\sum_{j=j_m+1}^\sN |\Csf_j|\| v_j\|_{L^2}\Big)^2.
\end{split}
\]
This yields
\[
\begin{split}
\| u^\sN(t)\|_{L^2(\Omega)}^2 &\geq \sum_{j=j_m}^\sN \Csf_j^2 e^{2\lambda_j t}\| v_j\|_{L^2}^2 -e^{(\lambda_{j_m+1}+\lambda_{j_m+2})t} \Big(\sum_{j=j_m+1}^\sN |\Csf_j| \| v_j\|_{L^2}\Big)^2 \\
&\qquad - |\Csf_{j_m}| e^{(\lambda_{j_m}+\lambda_{j_m+1})t}\| v_{j_m}\|_{L^2} \Big(\sum_{j=j_m+1}^\sN |\Csf_j|  \| v_j\|_{L^2}\Big).
\end{split}
\]
Due to the assumption \eqref{2ndCondC}, we deduce that
\[
\begin{split}
\| u^\sN(t)\|_{L^2}^2 &\geq \sum_{j=j_m}^\sN \Csf_j^2 e^{2\lambda_j t}\| v_j\|_{L^2}^2  - \frac14 \Csf_{j_m}^2 e^{(\lambda_{j_m+1}+\lambda_{j_m+2})t} \| v_{j_m}\|_{L^2}^2\\
&\qquad- \frac12 \Csf_{j_m}^2 e^{(\lambda_{j_m}+\lambda_{j_m+1})t} \| v_{j_m}\|_{L^2}^2.
\end{split}
\]
This yields 
\[\begin{split}
\| u^\sN(t)\|_{L^2}^2 &\geq \Csf_{j_m}^2\Big(e^{2\lambda_{j_m} t}- \frac12 e^{(\lambda_{j_m}+\lambda_{j_m+1})t} - \frac14 e^{(\lambda_{j_m+1}+\lambda_{j_m+2})t}\Big) \| v_{j_m}\|_{L^2}^2  \\
&\qquad\qquad +\sum_{j=j_m+1}^\sN \Csf_j^2 e^{2\lambda_j t}\| v_j\|_{L^2}^2.
\end{split}
\]
Notice that for all $t\geq 0$,
\[
e^{2\lambda_{j_m} t}-\frac12 e^{(\lambda_{j_m}+\lambda_{j_m+1})t} - \frac14 e^{(\lambda_{j_m+1}+\lambda_{j_m+2})t} \geq \frac14 e^{2\lambda_{j_m} t}.
\]
Hence, we have 
\begin{equation}\label{L^2NormU_2^M}
\| u^\sN(t)\|_{L^2} \geq C_5 F_\sN(t),
\end{equation}
for all $t\leq \min(T^\delta, T^\star, T^{\star\star})$.

Let 
\[
\tilde \Csf(\sN) = \max_{\sK+1\leq j\leq \sN}\frac{|\Csf_j|}{|\Csf_{j_m}|}\geq 0.
\]
We recall the definition of $T^{\star}$ and $T^{\star\star}$ from \eqref{DefTstar} and the fact that $T^{\delta}$ satisfies uniquely $\delta F_\sN(T^\delta)=\epsilon_0$, 
provided that  $\epsilon_0$ is taken to be 
\begin{equation}\label{EqEpsilon}
\epsilon_0 < \min\Big( \frac{C_2\delta_0}{C_3},  \frac{C_2^2}{2C_4(1+(\sN-\sK) \tilde \Csf(\sN))^3 }, \frac{C_5^2}{4C_4(1+ (\sN-\sK) \tilde \Csf(\sN))^3}  \Big).
\end{equation}
We prove that 
\begin{equation}\label{T_deltaMin}
T^{\delta} \leq \min\{T^{\star}, T^{\star\star}\}.
\end{equation}
In fact, if $T^{\star} < T^{\delta}$, we have from \eqref{NormUdelta} that
\[
\cE((\sigma^{\delta}, u^{\delta})(T^{\star})) \leq C_3 \delta F_\sN(T^{\star}) \leq C_3\delta F_\sN(T^{\delta}) = C_3\epsilon_0 <C_2\delta_0.
\]
And if $T^{\star\star}<T^{\delta}$, we have by \eqref{BoundDiff} and the definition of $C_2$ \eqref{DefC_12} that
\begin{equation}\label{EstBoundSigmaDelta}
\begin{split}
\|(\theta^{\delta},  u^{\delta})(T^{\delta})\|_{L^2} &\leq \delta \|(\theta^\sN,  u^\sN)(T^{\delta})\|_{L^2} + \|(\theta^d,  u^d)(T^{\delta})\|_{L^2}  \\
&\leq C_2 \delta F_\sN(T^{\delta}) + \sqrt{C_4}  \delta^{3/2} \Big( \sum_{j=j_m}^\sK |\Csf_j| e^{\lambda_j t} + \sum_{j=\sK+1}^\sN  |\Csf_j| e^{\frac23\Lambda t}\Big)^{3/2}.
\end{split}
\end{equation}
This implies 
\[
\begin{split}
\|(\theta^{\delta},  u^{\delta})(T^{\delta})\|_{L^2} &\leq C_2\delta F_\sN(T^{\delta}) + \sqrt{C_4}(1+ (\sN-\sK) \tilde \Csf(\sN))^{3/2}\delta^{3/2} F_\sN^{3/2}(T^\delta) \\
&\leq C_2\epsilon_0+  \sqrt{C_4}(1+(\sN-\sK) \tilde \Csf(\sN))^{3/2}\epsilon_0^{3/2}.
\end{split}
\]
Using \eqref{EqEpsilon} again, we deduce 
\[
\|(\theta^{\delta},  u^{\delta})(T^{\delta})\|_{L^2}< 2 C_2 \epsilon_0 = 2C_2 \delta F_\sN(T^{\delta}).
\]
which also contradicts  the definition of $T^{\star\star}$. 

Once we have \eqref{T_deltaMin}, we obtain from \eqref{BoundDiff} and \eqref{L^2NormU_2^M} that 
\[
\begin{split}
\| u^{\delta}(T^{\delta})\|_{L^2} &\geq \delta \| u^\sN(T^{\delta})\|_{L^2} - \| u^d(T^\delta)\|_{L^2}\\
&\geq C_5 \delta F_\sN(T^{\delta}) - \sqrt{C_4}  \delta^{3/2} \Big( \sum_{j=j_m}^\sK |\Csf_j| e^{\lambda_j t} +  \sum_{j=\sK+1}^\sN  |\Csf_j| e^{\frac23 \Lambda t}\Big)^{3/2}.
\end{split}
\]
Therefore, 
\begin{equation}
\begin{split}
\| u^{\delta}(T^{\delta})\|_{L^2} &\geq C_5 \epsilon_0 - \sqrt{C_4}(1+ (\sN-\sK) \tilde \Csf(\sN))^{3/2}\epsilon_0^{3/2} \geq \frac{C_5\epsilon_0}2 >0.
\end{split}
\end{equation}
The inequality \eqref{BoundU_2Tdelta} is proven by taking $\delta_0$ satisfying Proposition \ref{PropEstimates},  $\epsilon_0$ satisfying \eqref{EqEpsilon} and $m_0=C_5/2$. This ends the proof of Theorem \ref{ThmUnstable}.
\end{proof}

\end{document}